\documentclass[a4paper,10pt]{article}
\usepackage{amssymb,amsmath, amsthm, amsfonts}
\usepackage{mathrsfs}
\usepackage[all]{xypic}
\usepackage{tkz-graph}
\usetikzlibrary{decorations.markings}

\usepackage[left=4cm,right=4cm,top=2.5cm,bottom=2.5cm]{geometry}

\newtheorem{theorem}{Theorem}[section]
\newtheorem{lemma}[theorem]{Lemma}
\newtheorem{corollary}[theorem]{Corollary}
\newtheorem{proposition}[theorem]{Proposition}
\newtheorem{claim}{Claim}

\theoremstyle{definition}
\newtheorem{definition}[theorem]{Definition}
\newtheorem{remark}[theorem]{Remark}
\newtheorem{example}[theorem]{Example}
\newtheorem{notation}{Notation}

\newcommand{\free}{{\rm Free}}
\newcommand{\bF}{{\bf F}}

\newcommand{\cleq}{{\leq}}
\newcommand{\cgeq}{{\geq}}

\begin{document}

\title{Algebras and relational frames for G\"odel modal logic and some of its extensions}
\author{Tommaso Flaminio$^1$,  Lluis Godo$^1$, \\ Paula Mench\'on$^2$,  Ricardo O. Rodriguez$^3$\\
{\small $^1$ Artificial Intelligence Research Institute (IIIA - CSIC). Barcelona, Spain. }\\
{\small{\tt \{tommaso,godo\}@iiia.csic.es}}\\
{\small $^2$ CONICET, Univ. Nacional del Centro de la Provincia de Buenos Aires. Tandil, Argentina}\\
{\small {\tt paulamenchon@gmail.com}}\\
{\small $^3$UBA-FCEyN, Departamento de Computaci\'on,}\\
{\small CONICET-UBA, Inst. de Invest. en Cs. de la Computaci\'on. Buenos Aires, Argentina.}\\
{\small {\tt ricardo@dc.uba.ar}}}
\date{}

\maketitle


\begin{abstract}
G\"odel modal logics can be seen as extenions of intutionistic modal logics with the prelinearity axiom. In this paper we focus on the algebraic and relational semantics for G\"odel modal logics that leverages on the duality between finite G\"odel algebras and finite forests, i.e. finite posets whose principal downsets are totally ordered. We consider different subvarieties of the basic variety $\mathbb{GAO}$ of G\"odel algebras with two modal operators (GAOs for short) and their corresponding classes of forest frames, either with one or two accessibility relations. These relational structures can be considered as prelinear versions of the usual relational semantics of intuitionistic modal logic.  More precisely we consider two main extensions of finite G\"odel algebras with operators:  the one obtained by adding Dunn axioms, typically studied in the fragment of positive classical (and intuitionistic) logic,  and the one determined by adding Fischer Servi axioms. We present J\'onsson-Tarski like representation theorems for the different types of finite GAOs considered in the paper. 
\end{abstract}

\section{Introduction}
Extending modal logics to a non-classical propositional ground has been, and still is, a fruitful research line that encompasses several approaches, ideas and methods. In the last years, this topic  has significantly impacted on the community of many-valued and mathematical fuzzy logic that have proposed ways to expand fuzzy logics (t-norm based fuzzy logics, in the terminology of H\'ajek \cite{H98}) by modal operators so as to capture modes of truth that can be faithfully described as ``graded''. 

In this line, one of the fuzzy logics that has been an object of major interest without any doubt is the so called {\em G\"odel logic}, i.e., the axiomatic extension of intuitionistic propositional calculus given by the {\em prelinearity axiom}: $(\varphi\to \psi)\vee(\psi\to \varphi)$. As first observed by Horn in \cite{Horn}, prelinearity implies completeness of G\"odel logic with respect to totally ordered Heyting algebras, i.e., {\em G\"odel chains}. Indeed, prelinear Heyting algebras form a proper subvariety of that of Heyting algebras, usually called the variety of G\"odel algebras and denoted  $\mathbb{G}$ whose subdirectly irreducible elements are totally ordered. Furthermore, in contrast with the intuitionistic case, $\mathbb{G}$ is locally finite, whence the finitely generated free G\"odel algebras are finite.

Modal extensions of G\"odel logic have been intensively discussed in the literature \cite{CMRR,XARO,RodVid}. Following the usual methodological and philosophical approach to fuzzy logic, they have been mainly approached semantically by generalizing the classical definition of Kripke model $\langle W, R, e\rangle$ by allowing both the evaluation $e$ of (modal) formulas and the accessibility relation $R$ to range over a G\"odel algebra, rather than the classical two-valued set $\{0,1\}$ (see \cite{BEGR} for a general approach).  More precisely, a model of this kind, besides evaluating formulas in a more general structure than the classical two-element boolean algebra, regards the accessibility relation $R$ as a function from the cartesian product $W\times W$ to a Godel algebra ${\bf A}$ so that, for all $w,w'\in W$, $R(w,w')=a\in A$ means that $a$ is the {\em degree of accessibility} of $w'$ from $w$. 

In this chapter we will put forward a novel approach to G\"odel modal logic that leverages on the duality between finite G\"odel algebras and finite forests. This approach, preliminary presented in \cite{FGR}, will be deepened and extended in the present paper. In fact, 
since G\"odel algebras are nothing but prelinear Heyting algebras, and their lattice reducts are distributive lattices, there are a number of previous  works  in the literature, both for distributive modal algebras (see e.g. 
\cite{Viorica00, Viorica00-2, Goldblatt, Cel06, Cel08, Petrovich}) and for Heying modal algebras or modal intuitionistic logics (see e.g. 
\cite{Ono, FS, Sotirov, BoDo, Dosen, PlotStir, Dunn, Cel01, Ale, Orlo, CelJan}), from which many results can be adapted to our setting. 

In particular, we will focus on G\"odel modal algebras and their dual structures, that is, the prime spectra of G\"odel algebras ordered by reverse-inclusion. These ordered structures can be regarded as the prelinear version of posets and they are known in the literature as {\em forests}: posets whose principal downsets are totally ordered.  The algebras we will consider 
form a variety  denoted by $\mathbb{GAO}$ for {\em G\"odel algebras with operators}.  Hence, the algebras we are concerned with are those belonging to the finite slice of $\mathbb{GAO}$. The associated relational structures based on forests, as we briefly recalled above, might hence be regarded as the prelinear version of the usual relational semantics of intuitionistic modal logic. Accessibility relations $R_\Box$ and $R_\Diamond$ on finite forests are defined, in our frames, by ad hoc properties  that we  express in terms of (anti)monotonicity on the first argument of the relations themselves. These relational frames will be called {\em forest frames}.

In this chapter we will be mainly concerned with J\'onsson-Tarski like representation theorems for  G\"odel algebras with operators (and some of their extensions) as already done for the Boolean case \cite{JT61,Lemmon}, the Heyting case \cite{Ale, Orlo} and the case of their common positive fragment \cite{CelJan}. It is worth noticing that, although J\'onsson-Tarski like representation theorems for  modal algebras might be proved without making an explicit reference to the relational frames that determine the isomorphic copy of the starting algebra (see for instance \cite{Lemmon,BoDo,JT61}), in the more recent papers \cite{Ale,Orlo}, these relational frames are explicitly used in the proof of such theorems. In this chapter we will follow this latter approach as it will give us also the opportunity of pointing out which relational frames are more or less general  in a sense that will be made clear in Section \ref{sec:tworelations}.
Along the whole chapter, we will take care of comparing our approach to GAOs with the ones that underly studies on minimal intuitionistic modal logics, in particular those developed by Bo\v{z}i\'c and Do\v{s}en in \cite{BoDo}, that later reappear  in the work of 
Palmigiano \cite{Ale} and Or{\l}owska and Rewitzky  \cite{Orlo}.  Note that in this chapter we will not deal with logic at all.


More in detail, we will observe that, if we start from any G\"odel algebra with operators $({\bf A}, \Box,\Diamond)$, its associated forest frame $({\bf F}_{\bf A}, R_\Box, R_\Diamond)$ allows to construct another algebraic structure $({\bf G}(\bF_{\bf A}), \beta_\Box, \delta_\Diamond)$ isomorphic to the starting one. Interestingly, the forest frame $({\bf F}_{\bf A}, R_\Box, R_\Diamond)$ is not the unique one that reconstructs $({\bf A}, \Box, \Diamond)$ up to isomorphism. Indeed, as we will show in Section \ref{sec:tworelations}, for every G\"odel algebra with operators $({\bf A}, \Box,\Diamond)$, there are not isomorphic forest frames, Palmigiano-like and Or{\l}owska and Rewitzky-like frames that determine the same original modal algebra $({\bf A}, \Box, \Diamond)$ up to isomorphism. 

In Section \ref{GAO} we will start by considering the most general way to define the operators $\Box$ and $\Diamond$ on G\"odel algebras while in Section \ref{sec:tworelations} we investigate the relational structures corresponding to the resulting algebraic structures. In Section \ref{extensions} we will focus on particular and well-known extensions. Precisely we will consider two main extensions of G\"odel algebras with operators: (1) the first one is obtained by adding the Dunn axioms, typically studied in the fragment of positive classical (and intuitionistic) logic \cite{Dunn,CelJan}; (2) the second one is determined by  adding the Fischer Servi axioms \cite{FS}. From the algebraic perspective, adding these  identities to G\"odel algebras with operators identifies two proper subvarieties of $\mathbb{GAO}$ that will be respectively denoted by $\mathbb{DGAO}$ and $\mathbb{FSGAO}$. 
Section \ref{extensions} is hence complemented by some examples that showing that $\mathbb{DGAO}$ and $\mathbb{FSGAO}$ can be distinguished.  

In contrast with the case of general G\"odel algebras with operators discussed in Sections \ref{GAO} and \ref{sec:tworelations} whose relational structures need two independent relations to treat the modal operators, the structures belonging to $\mathbb{DGAO}$ and $\mathbb{FSGAO}$ only need, for their J\'onnson-Tarski like representation, frames with only one accessibility relation. Forest frames with one relation are hence studied in Section \ref{one-relation} where, in addition to a comparison with the usual intuitionistic case, we will also study in detail the relational structures corresponding to two further subvarieties of $\mathbb{GAO}$. The first one is the variety $\mathbb{FSDGAO}$ obtained as the intersection  $\mathbb{DGAO}\cap\mathbb{FSGAO}$. The algebras belonging to such variety have been called {\em bi-modal G\"odel algebras} in \cite{XARO}.
The second subvariety that we will consider in Section \ref{MDGAO} is another refinement of $\mathbb{DGAO}$ and it will be denoted by $\mathbb{WGAO}$. Algebras in this class are characterized by the requirement that $\Box a$ and $\Diamond a$ are Boolean for of each element $a$. A final proposition will make clear the inclusions between subvarieties of $\mathbb{GAO}$ studied in this paper.
 
 Next section on preliminaries is devoted to introduce the basics of finite G\"odel algebras and their dual structures of finite forests. 

\section{Preliminaries: G\"odel algebras and forests}\label{sec:algandFortests}
G\"odel algebras, the algebraic semantics of infinite-valued G\"odel logic \cite{H98}, are    idempotent, bounded, integral, commutative residuated lattices of the form ${\bf A}=(A, \wedge,\vee, \to, \bot,\top)$ satisfying the prelinearity equation: $(a\to b)\vee(b\to a)=\top$. In other words, G\"odel algebras are {\em prelinear Heyting algebras}.

Let us recall that  the prelinearity equation has a twofold effect on Heyting algebras. Indeed, from the logical side it makes G\"odel logic to be sound and complete w.r.t. totally ordered truth-value scales and hence it properly presents it as a fuzzy logic in the sense of \cite{H98}. Also, from the universal algebraic perspective, it makes the variety of G\"odel algebras locally finite, whereas as it is well-known, Heyting algebras are not. This latter observation is particularly important for us because it will allows us to provide relevant examples of G\"odel algebras with operators based on finite freely generated structures. This latter fact, for clear reasons, is not possible in the Heyting realm. In fact, all algebras we will consider in this paper will be assumed to be finite.

Another important feature of G\"odel logic, that spotlights a nice behavior compared to other well-known many-valued logics, follows from an observation made by Takeuti and Titani \cite{Take} that characterises G\"odel implication operator  $\to$ as that unique truth-function on $[0,1]$ satisfying natural properties relating $\to$ with the order of $[0,1]$ and, very importantly, satisfying the classical deduction theorem. See \cite{Take} and \cite{BaazPre} for more insights on the subject.

Let ${\bf A}$ be a G\"odel algebra. A non-empty subset $f$ of ${\bf A}$ is said to be a {\em filter} provided that: (1) $\top\in f$, (2) if $x,y\in f$, then $x\wedge y\in f$, (3) if $x\in f$ and $y\geq x$ then $y\in f$. A filter $f\neq A$ (that is a {\em proper} filter) is said to be {\em prime} if $x\vee y\in f$ implies that either $x\in f$ or $y\in f$. A filter $f$ is {\em principal} (or {\em principally generated}) if there exists an element $x\in A$ such that $f={\uparrow}x=\{y\in A\mid y\geq x\}$. By a standard result \cite[]{}, in every finite G\"odel algebra prime filters coincide with those filters principally generated by the join-irreducible elements of $A$.

 A non-empty subset  $h$ of $A$ is said to be a {\em co-filter}, if (1) if $x\in h$ and $y\geq x$ implies $y\in h$ and (2) $x\vee y\in f$ implies that either $x\in h$ or $y\in h$. Therefore a subset $f$ of $A$ is a prime filter iff it is both a filter and a co-filter.

 A non-empty subset $i$ of $A$ is an {\em ideal} provided that: (1) $i$ is downward closed and such that, if $x, y\in i$ then $x\vee y\in i$. It is easy to see that the set-theoretical complement of a proper co-filter is an ideal.

Let ${\bf A}$ be a finite G\"odel algebra and denote by $F_{\bf A}$ the finite set of its prime filters. Unlike the case of boolean algebras, prime and maximal filters are not the same for G\"odel algebras and indeed $F_{\bf A}$ can be ordered in a nontrivial way. In particular, if for $f_1, f_2\in F_{\bf A}$ we define $f_1\leq f_2$ iff (as prime filters) $f_1\supseteq f_2$,  ${\bf F}_{\bf A}=(F_{\bf A}, \leq)$ turns out to be a finite {\em forest}, i.e., a poset such that the downset of each element is totally ordered.

Finite forests  play a crucial role in the theory of finite G\"odel algebras. Indeed, let ${\bf F}=(F,\leq)$ be a finite forest,  $G(\bF)$ be the set of all downward closed subsets of $F$ (i.e., the {\em subforests} of ${\bf F}$) and consider the following operations on $G(\bF)$: for all $x,y\in F$,
\begin{itemize}
\item[1.] $x\wedge y= x\cap y$ (the set-theoretic intersection);
\item[2.] $x\vee y= x\cup y$ (the set-theoretic union);
\item[3.] $x\to y= ({\uparrow}(x\setminus y))^c=F\setminus {\uparrow}(x\setminus y)$, where $\setminus$ denotes the set-theoretical difference, for every $z\in F$, ${\uparrow}z=\{k\in F\mid k\geq z\}$ and $^c$ denotes the set-theoretical complement.\footnote{Without danger of confusion, and thanks  to the following result, we will not distinguish the symbols of a G\"odel algebra ${\bf A}$ from those of ${\bf G}({\bF})$} 
\end{itemize}
The algebra ${\bf G}(\bF)=(G(\bF), \wedge, \vee, \to, \emptyset, F)$ is a G\"odel algebra \cite[\S4.2]{ABG} 
 and  the following Stone-like representation theorem holds.

\begin{lemma}[{\cite[Theorem 4.2.1]{ABG}}]\label{lemma1}
Every finite  G\"odel algebra ${\bf A}$ is isomorphic to ${\bf G}(\bF_{\bf A})$ through the map $r: {\bf A}\to {\bf G}(\bF_{\bf A})$
$$
r: a\in A\mapsto\{f\in F_{\bf A}\mid a\in f\}.
$$
\end{lemma}

\begin{example}\label{example1}
Let $\free_1$ be the 1-generated free G\"odel algebra  (Fig. \ref{fig1}). Its prime filters, which are all principally generated as upsets of its join-irreducible elements, are $f_1=\{y\in \free_1\mid y\geq x\} = \{x, x \lor \neg x, \neg\neg x, \top\}$, $f_2=\{y\in \free_1\mid y\geq \neg x\} =\{ \neg x, x \lor \neg x, \top\}$, and $f_3=\{y\in \free_1\mid y\geq \neg \neg x\} = \{\neg\neg x, \top\}$. The forest $\bF_{\free_1}$ is obtained by ordering $\{f_1, f_2, f_3\}$ by reverse inclusion.

Let us consider the set $G({\bF_{\free_1}})$ of subforests of $\bF_{\free_1}$:
$$
G({\bF_{\free_1}})=\{\emptyset, F_{\free_1}, \{f_2\}, \{f_1\}, \{f_2, f_1\}, \{f_3, f_1\}\}
$$
with operations $\wedge, \vee, \to$ as in (1-3) above. Lemma \ref{lemma1}  shows that algebra ${\bf G}({\bF_{\free_1}})$ is a G\"odel algebra which is isomorphic to $\free_1$.
\begin{figure}
\begin{center}
\begin{tikzpicture}

  \node [label=below:{${\bot}$}, label=below:{} ] (n1)  {} ;
  \node [above right of=n1,label=right:{${x}$}, label=below:{ }] (n2)  {} ;
  \node [above left of=n1,label=left:{${\neg x}$},label=below:{}] (n3) {} ;
  \node [above right of=n2,label=right:{${\neg \neg x}$},label=below:{}] (n4) {} ;
  \node [above left of=n4,label=above:{${\top}$},label=below:{}] (n6) {} ;
  \node [above right of=n3,label=left:{${\neg x \vee x}$},label=below:{}] (n5) {} ;

  \draw  (n1) -- (n2);
  \draw (n1) -- (n3);
  \draw  (n3) -- (n5);
  \draw  (n2) -- (n4);
  \draw  (n4) -- (n6);
  \draw  (n5) -- (n6);
  \draw  (n2) -- (n5);
  \draw [fill] (n1) circle [radius=.5mm];
  \draw [fill] (n2) circle [radius=.5mm];
  \draw [fill] (n3) circle [radius=.5mm];
  \draw [fill] (n4) circle [radius=.5mm];
  \draw [fill] (n5) circle [radius=.5mm];
  \draw [fill] (n6) circle [radius=.5mm];

\end{tikzpicture}
%
\hspace{.6cm}
\begin{tikzpicture}
  \node [label=below:{$f_2$}, label=below:{} ] (n1)  {} ;
  \node [right of=n1,label=below:{$f_1$}, label=below:{} ] (n2)  {} ;
  \node [above of=n2, label=above:{$f_3$}, label=below:{} ] (n3)  {} ;
  \draw [fill] (n1) circle [radius=.5mm];
    \draw [fill] (n2) circle [radius=.5mm];
      \draw [fill] (n3) circle [radius=.5mm];
   \draw (n2)--(n3);
\end{tikzpicture}
\hfill
%
\begin{tikzpicture}

  \node [label=below:{${\emptyset}$}, label=below:{} ] (n1)  {} ;
  \node [above right of=n1,label=right:{${\{f_1\}}$}, label=below:{ }] (n2)  {} ;
  \node [above left of=n1,label=left:{${\{f_2\}}$},label=below:{}] (n3) {} ;
  \node [above right of=n2,label=right:{${\{f_1, f_3\}}$},label=below:{}] (n4) {} ;
  \node [above left of=n4,label=above:{${F_{\free_1}}$},label=below:{}] (n6) {} ;
  \node [above right of=n3,label=left:{${\{f_1, f_2\}}$},label=below:{}] (n5) {} ;

  \draw  (n1) -- (n2);
  \draw (n1) -- (n3);
  \draw  (n3) -- (n5);
  \draw  (n2) -- (n4);
  \draw  (n4) -- (n6);
  \draw  (n5) -- (n6);
  \draw  (n2) -- (n5);
  \draw [fill] (n1) circle [radius=.5mm];
  \draw [fill] (n2) circle [radius=.5mm];
  \draw [fill] (n3) circle [radius=.5mm];
  \draw [fill] (n4) circle [radius=.5mm];
  \draw [fill] (n5) circle [radius=.5mm];
  \draw [fill] (n6) circle [radius=.5mm];

\end{tikzpicture}
\end{center}
\caption{{\small From left to right: The Hasse diagram of the free G\"odel algebra over one generator $\free_1$, the forest $\bF_{\free_1}$ of its prime filters, and the Hasse diagram of its isomorphic copy ${\bf G}({\bF_{\free_1}})$}}
\label{fig1}
\end{figure}
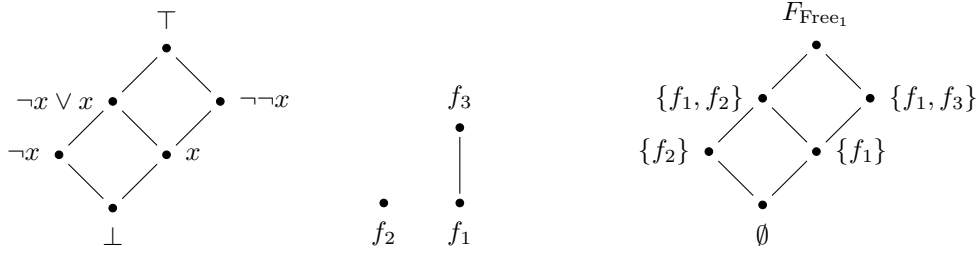
\end{example}

\section{G\"odel algebras with operators} \label{GAO}


In this section we introduce the basic class of G\"odel algebras with two modal operators that we will deal with in this chapter, and prove a representation theorem for them \`a la  J\'onsson-Tarski. 


Before entering into the details of the representation theorem we will show, it is interesting to recall that J\'onsson-Tarski theorem \cite{JT61} allows to isomorphically represent a Boolean algebra with operators (BAOs) as a another BAO on a field of sets (a subalgebra of a power set). Moreover, in Lemmon's celebrated paper \cite{Lemmon}, it is shown that for every {\em finite} BAO $({\bf A}, \Box, \Diamond)$ there exists another  BAO $({\bf A}', \Box', \Diamond')$, built from the set of prime filters of ${\bf A}$, such that $({\bf A}, \Box, \Diamond)$ and $({\bf A}', \Box', \Diamond')$ are isomorphic.

Results in the direction of giving similar representations for modal algebras, have been provided by several scholars. We here recall just a few, and in particular \cite{CelJan} for  the case of positive modal algebras, \cite{Cel08} for distributive lattices with operators, and \cite{Orlo,Ale} for Heyting algebras with operators. 

As for the particular case of Heyting algebras with operators (HAOs), similar representation results have been proven. More precisely, we can identify two types of representation theorems for (not finite in general) HAOs. The first kind provides, starting from any HAO $({\bf A}, \Box, \Diamond)$, another HAO $({\bf A}', \Box', \Diamond')$ in which the first one embeds (see for instance \cite{Orlo}); the second kind uses topological  techniques to construct $({\bf A}', \Box', \Diamond')$ that, more in the style of  J\'onsson-Tarski result, turns out to be isomorphic to the starting one (see \cite{Ale}). 

Unlike the case of BAO's, however, few papers consider the case of finite algebras and, although both techniques used to determine the embedding and the isomorphism just recalled should easily apply to this finite case, no explicit proof has been presented, as far as the authors knows.   This is the main reason why, in this section, we decided to make clear all the  steps needed to prove our main result, namely Theorem \ref{thm:isoBox} below whose statement could reasonably look familiar to some reader. Let us stress once again that our algebraic setting of finite G\"odel algebras lies, in terms of generality, between finite Boolean algebras with operators and finite Heyting algebras with operators.

Next definition is from \cite[Definition 5]{FGR}. 
 
\begin{definition}\label{def:GAO}
A {\em G\"odel algebra with operators} ({\em GAO} for short) is a triple $({\bf A}, \Box, \Diamond)$ where ${\bf A}$ is a G\"odel algebra, $\Box$ and $\Diamond$ are unary operators on $A$ satisfying the following equations:
\begin{itemize}
\item[$(\Box1)$] $\Box\top=\top$;
\item[$(\Box2)$] $\Box(x\wedge y)=\Box x\wedge \Box y$;
\item[$(\Diamond1)$] $\Diamond\bot=\bot$;
\item[$(\Diamond2)$] $\Diamond(x\vee y)=\Diamond x\vee \Diamond y$.
\end{itemize}
\end{definition}
Clearly the class of G\"odel algebras with operators forms a variety (i.e., an equational class) that we will henceforth denote by $\mathbb{GAO}$. 

 The class of GAOs in fact coincides with the class of  extensions with the prelinearity axiom of the algebras associated to the intuitionistic modal logic $IK(\Box, \Diamond)$ considered by Sotirov in \cite{Sotirov}, which in turn coincide with the so-called HK-algebras in \cite{Orlo}. Moreover, the fragment of GAOs without implication also coincide with the so-called $\Box\Diamond$-lattices in \cite{Cel06,Cel08}.
 
Although it could be presented as an adaptation for GAOs of what is proved in \cite{Orlo} for HK-algebras, in what follows we will provide all the needed details and proofs for a representation theorem for GAOs (next Theorem \ref{thm:isoBox}).

Let us start by showing some easy properties which will turn out to be useful for the rest of this section.

\begin{proposition}\label{prop:basic}[C.f. \cite[Lemmas 3.1 and 3.2]{Orlo}]
For every GAO $({\bf A}, \Box, \Diamond)$ and for every filter $f$ of ${\bf A}$ the following facts hold.
\begin{enumerate}
\item $\Box^{-1}(f)$ is a filter;
\item If $f$ is prime, then $\Diamond^{-1}(f)$ is a co-filter. 
\end{enumerate}
\end{proposition}
\begin{proof} 
(1) Since $f$ is a filter, $\top\in \Box^{-1}(f)$ because of $(\Box1)$; $\Box^{-1}(f)$ is $\wedge$-closed because of $(\Box2)$; $\Box^{-1}(f)$ is upward closed because $\Box$ is monotone.

(2) Since $f$ is prime,  $f$ is a co-filter. By $(\Diamond1)$, $\Diamond\bot=\bot$ and hence $\Diamond\bot\not \in f$ because $f$ is proper. Therefore the claim follows from \cite[Lemma 3.2 (a)]{Orlo}.
\end{proof}


Let $({\bf A},\Box, \Diamond)$ be a GAO and let $\bF_{\bf A}$ the forest of its prime filters. Define $R_\Box$ and $R_\Diamond$ on $F_{\bf A}\times F_{\bf A}$ as follows: 
for each $f_1, f_2\in F_{\bf A}$,
 \begin{equation}\label{eq:QBox}
f_1 R_\Box f_2\mbox{ iff } \Box^{-1}(f_1)\subseteq f_2 
 \end{equation}
 and 
 \begin{equation}\label{eq:Qdiamond}
f_1 R_\Diamond f_2\mbox{ iff }\Diamond(f_2)\subseteq f_1 \mbox{ iff } f_2 \subseteq \Diamond^{-1}(f_1). 
\end{equation}
 
 \begin{lemma}\label{lemmaMA}
 For every  GAO $({\bf A}, \Box, \Diamond)$, the relations $R_\Box$ and $R_\Diamond$ respectively satisfy
 \begin{itemize}
 \item[$(M)$]   for all $f, g, h\in F_{\bf A}$, if $f\leq g$ and $f R_\Box h$, then  $g R_\Box h$;
\item[$(A)$] for all $f, g, h\in F_{\bf A}$, if $g\leq f$ and $f R_\Diamond h$, then  $g R_\Diamond h$.
\end{itemize}
 \end{lemma}
 \begin{proof}
 Let $f,g,h\in F_{\bf A}$. If $f\leq g$ in the order of $\bF_{\bf A}$, then $f\supseteq g$ as prime filters, whence if $\Box^{-1}(f)\subseteq h$ then $\Box^{-1}(g)\subseteq h$. Therefore, if $f R_\Box h$, then $g R_\Box h$ that shows that $R_\Box$ satisfies $(M)$.
 
 As for the second claim, let $f,g,h\in F_{\bf A}$ and assume $f R_\Diamond h$ (i.e., $\Diamond(h)\subseteq f$ as prime filters) and $f\geq g$, meaning that, as prime filters, $f\subseteq g$. Then, $\Diamond(h)\subseteq f\subseteq g$ and hence $g R_\Diamond h$.
 \end{proof}

 Now, let $(\bF, R_\Box, R_\Diamond)$ be such that $\bF$ is a forest and $R_\Box, R_\Diamond\subseteq F\times F$ respectively satisfy $(M)$ and $(A)$ of Lemma \ref{lemmaMA}. Let  ${\bf G}({\bF})$ be the G\"odel algebra of downsets of $\bF$ defined as in the previous section 
and consider the maps $\beta, \delta: G(\bF)\to G(\bF)$ such that,
for every $a\in G(\bF)$

\begin{equation}\label{eqBetaR}
 \beta(a)=\{y\in F\mid \forall z\in F, \; (y R_\Box z\Rightarrow z\in a)\},
 \end{equation}
 and 
\begin{equation}\label{eqMQDiamond}
\delta(a)=\{y\in F\mid \exists z\in a,\; y R_\Diamond z\}.
\end{equation}
\begin{remark}
(1) For all $a\in G(\bF)$, $\beta(a)$ is a subforest of $\bF$. 
Indeed, if  $x\in \beta(a)$ then $\forall z\in F, \; (x R_\Box z \Rightarrow z\in a)$. Let $y\leq x$. For all $z$, if $y R_\Box z $, then $x R_\Box z $ as well, because of $(M)$, and hence $z\in a$. Thus $y\in \beta(a)$.
\vspace{.2cm}

(2) For all $a\in G(\bF)$,
$\delta(a)\in  G(\bF)$, i.e.,  $\delta(a)$ is a subforest of $\bF$. Indeed if $x\in \delta(a)$ then there exists $z\in a$ such that $x R_\Diamond z $. Let $y\leq x$ in $\bF$. Then  (A) of Lemma \ref{lemmaMA} implies $y R_\Diamond z$ as well, that is $y\in \delta(a)$ and hence $\delta(a)$ is downward closed. 
\end{remark}
Moreover, the following properties hold.
\begin{proposition}\label{prop:properties1}
Let ${\bf F}$ be a finite forest and let $R_\Box, R_\Diamond\subseteq F\times F$ such that $R_\Box$ satisfies $(M)$ and $R_\Diamond$ satisfies $(A)$. Let $\beta,\delta:G(\bF)\to G(\bF)$ be defined as in (\ref{eqBetaR}) and (\ref{eqMQDiamond}) respectively. Then:
\begin{enumerate}
\item $\beta(\top)=\top$;
\item  For all $a, b\in G(\bF)$, $\beta(a \land b)= \beta(a) \cap \beta(b)$.
\item $\delta(\bot)=\bot$;
\item For all $a, b\in G(\bF)$, $\delta(a \lor b)= \delta(a) \cup \delta(b)$.
\end{enumerate}
\end{proposition}
\begin{proof}
(1) Recall from Section \ref{sec:algandFortests} that the top element of ${\bf G}({\bF})$ is $F$. Thus, $\beta(\top)=\beta(F)=\{y\in F\mid \forall z\in F, \; (y R_\Box z \Rightarrow z\in F)\}$. Obviously, the condition $(y R_\Box z \Rightarrow z\in F)$ is true for all $z\in F$ and hence $\beta(F)=F$.
\vspace{.1cm}

\noindent(2)  For all $a,b \in G(\bF)$, we have, 
$$
\begin{array}{lll}
\beta(a \land b)&=&\{y\in F\mid \forall z\in F, \; y R_\Box z \Rightarrow z\in a \land b\}\\
&=& \{y\in F\mid \forall z\in F, \; y R_\Box z \Rightarrow z\in a \cap b\}\\
&=&\{y\in F\mid \forall z\in F, \; y R_\Box z \Rightarrow  z\in a \} \;\cap \\ & & \{y\in F\mid \forall z\in F, \; y R_\Box z \Rightarrow  z\in b \} \\
&=& \beta(a) \cap \beta(b).\\
\end{array}
$$
\vspace{.1cm}

\noindent(3) The bottom element of ${\bf G}({\bF})$ is the empty forest, whence $\emptyset=\{y\in F\mid \exists z\in \emptyset, \; y R_\Diamond z \}=\delta(\bot)$.
\vspace{.1cm}

\noindent (4) $\delta(a \lor b)=\{y\in F\mid \exists z\in a \lor b,\; y R_\Diamond z \} = \{y\in F\mid \exists z\in a \cup b,\; y R_\Diamond z \}=
\{y\in F\mid \exists z\in a,\;y R_\Diamond z \} \cup \{y\in F\mid \exists z\in b,\;y R_\Diamond z \} = \delta(a) \cup \delta(b)$.
\end{proof}

It was shown in Lemma \ref{lemma1} that, if $\bf A$ is a G\"odel algebra then ${\bf G}(\bF_{\bf A})$ is a G\"odel algebra too. Therefore, Prop. \ref{prop:properties1} shows that, if $({\bf A}, \Box, \Diamond)$ is a GAO, then $({\bf G}(\bF_{\bf A}), \beta, \delta)$ is a GAO as well. The following J\'onsson-Tarski like representation theorem shows they are isomorphic.

 \begin{theorem}\label{thm:isoBox}
 Every finite GAO $({\bf A}, \Box, \Diamond)$ is isomorphic  to the GAO $({\bf G}(\bF_{\bf A}), \beta, \delta)$ through the mapping $r: ({\bf A}, \Box, \Diamond)\to ({\bf G}(\bF_{\bf A}), \beta, \delta)$ such that, for every $a\in A$, $r(a) = \{f\in F_{\bf A}\mid a\in f\}$. In particular, for all $a\in A$,
\begin{equation}\label{eq2}
r(\Box(a))=\beta(r(a)) \mbox{ and }r(\Diamond (a))=\delta(r(a)).
\end{equation}
 \end{theorem}
 \begin{proof}
 We  showed in Lemma \ref{lemma1} that the map $r:{\bf A}\to{\bf G}(\bF_{\bf A})$ is a G\"odel isomorphism. Thus, it remains to show that (\ref{eq2}) holds. 
\vspace{.2cm}

\noindent(1) $r(\Box(a))=\beta(r(a))$.  Let us start proving that for all $a\in A$, $\beta(r(a))\subseteq r(\Box(a))$. By definition,
 $$
 \begin{array}{lll}
 \beta(r(a))&=&\{f\in F_{\bf A}\mid \forall g\in F_{\bf A}\;(f R_\Box g \Rightarrow g\in r(a))\}\\
 &=&\{f\in F_{\bf A}\mid \forall g\in F_{\bf A}\;(\Box^{-1}(f)\subseteq g\Rightarrow a\in g)\}.
 \end{array}
 $$
 Let $f\in \beta(r(a))$ and assume, by way of contradiction, that $f\not\in r(\Box(a))$, that is to say, $a\not\in \Box^{-1}(f)$. Notice that this assumption forces $a\neq \top$. 
 By Proposition \ref{prop:basic} (1), $\Box^{-1}(f)$ is a filter. Thus, 
%
 if $a\not\in \Box^{-1}(f)$ and since $a\neq\top$, by \cite[Lemma 2.3.15]{H98}, there exists a prime filter $g$ of ${\bf A}$ such that $g\supseteq \Box^{-1}(f)$ and $a\not\in g$. On the other hand, $f R_\Box g $ because $g$ extends $\Box^{-1}(f)$ and $a\not\in g$. Thus, $f\not\in \beta(r(a))$ and a contradiction has been reached.

 For the other inclusion, we have to prove that if $\Box(a)\in f$, then for all $g\in F_{\bf A}$, $f R_\Box g \Rightarrow a\in g$. If $\Box(a)\in f$, then $a\in \Box^{-1}(f)$. Therefore, for all $g\in F_{\bf A}$, if $f R_\Box g $, then $\Box^{-1}(f)\subseteq g$ and hence $a\in g$ which settles the claim.

\vspace{.2cm}

\noindent(2) $r(\Diamond (a))=\delta(r(a))$. 
First of all notice that it is sufficient to prove it for the case of $a$ being a join-irreducible element of ${\bf A}$. Indeed, assume that the right-hand-side of (\ref{eq2}) holds for join irreducible elements and let $b$ be not join irreducible. Then $b$ can be displayed as $b=a_1\vee\ldots\vee a_k$, where  the $a_i$'s are join irreducible. By $(\Diamond2)$, $\Diamond(b)=\Diamond(a_1)\vee\ldots\vee\Diamond(a_k)$. Therefore, since $r$ is a G\"odel algebra isomorphism,
$$
r(\Diamond(b))=r(\Diamond(a_1))\vee\ldots\vee r(\Diamond(a_k)).
$$
By assumption, $r(\Diamond a_i)=\delta(r(a_i))$ for all $i=1,\ldots,k$. Thus, $r(\Diamond(b))=\delta(a_1)\vee\ldots\vee \delta(a_k)$ which equals $\delta(b)$ by Proposition \ref{prop:properties1}(2).

Let hence $a$ be join irreducible and let us prove that $r(\Diamond (a))\supseteq\delta(r(a))$ and $r(\Diamond (a))\subseteq\delta(r(a))$. As for the first inclusion, notice that for all $a\in A$ (being $a$ join irreducible or not), by Lemma \ref{lemma1}, 
$$
\begin{array}{lll}
\delta(r(a))&=&\{f\in F_{\bf A}\mid \exists g\in r(a),\; f R_\Diamond g  \}\\
&=&\{f\in F_{\bf A}\mid \exists g\in F_{\bf A},\; (a\in g\;\& \; f R_\Diamond g  \}\\
&=&\{f\in F_{\bf A}\mid \exists g\in F_{\bf A},\; (a\in g\;\& \; \Diamond(g)\subseteq f)\}.
\end{array}
$$

Therefore, if $f\in \delta(r(a))$, $\Diamond(a)\in f$ and hence $f\in r(\Diamond(a))$ and hence $r(\Diamond (a))\supseteq\delta(r(a))$.

To prove the other inclusion we have to show that if $f'\in r(\Diamond(a))$,  there exists an $f\in F_{\bf A}$ such that $a\in f$ and $\Diamond(f)\subseteq f'$.
Since  $a$ is join irreducible, the filter $f_a=\{b\in A\mid b\geq a\}$ is prime. Let us prove that $\Diamond(f_a)\subseteq f'$.
\begin{claim}\label{claim0}
$\Diamond(f_a)\subseteq f_{\Diamond(a)}=\{x\in A\mid x\geq \Diamond(a)\}$.
\end{claim}
As a matter of fact, if $z\in \Diamond(f_a)$, then there exists $b\geq a$ such that $z=\Diamond(b)$. Since $\Diamond$ is monotone, $\Diamond(b)\geq \Diamond(a)$, whence $z=\Diamond(b)\in f_{\Diamond(a)}$.

\begin{claim}\label{claim1}
For all $f'\in r(\Diamond(a))$,
$f_{\Diamond(a)}\subseteq f'$.
\end{claim}
Indeed, if  $x\in f_{\Diamond(a)}$, then $x\geq \Diamond(a)$ and hence $x\in f'$ because $\Diamond(a)\in f'$ and  $f'$ is  upward closed.

By the above claims, for all $f'\in r(\Diamond(a))$, $\Diamond(f_a)\subseteq f'$, whence
$$
r(\Diamond(a))\subseteq \delta(r(a)).
$$
Thus, for all $a$, $r(\Diamond(a))= \delta(r(a))$ which settles the claim.
 \end{proof}
 Theorem \ref{thm:isoBox} above shows that every finite GAO can be isomorphically represented as the algebra of subforests of the forests of its prime filters. This kind of representation   will be henceforth call {\em forest-based representation}. In these latter algebras, the modal operators are obtained by two  binary relations $R_\Box$ and $R_\Diamond$ satisfying $(M)$ and $(A)$ respectively. 
Let us further notice that, although these two relations $R_\Box$ and $R_\Diamond$ are independent in general, there are significant cases in which they are not, or  they can even coincide. An example of the latter is the case of classical Kripke frames, the dual semantics of Boolean algebras with operators, a proper subvariety of $\mathbb{GAO}$.

We will deepen the investigation on such relational models in the next section, but it is worth pointing out that, indeed, the theorem above is sufficiently general to be rephrased for every subclass $\mathbb{K}$ of $\mathbb{GAO}$ that is closed under isomorphic images. i.e., such that $\mathbb{K}={\bf I}(\mathbb{K})$. Thus, in particular, it applies to all subvarieties of $\mathbb{GAO}$. The following easy consequence of Theorem \ref{thm:isoBox} makes this fact clear.

 \begin{corollary}\label{corollarySubclass}
 Let $\mathbb{K}$ be any subset of $\mathbb{GAO}$ that is closed under isomorphic images. Then, every algebra in $\mathbb{K}$ has an isomorphic forest-based representation in $\mathbb{K}$.
 \end{corollary}
 
 Note that, in particular, this corollary will be applicable to the four subvarieties of $\mathbb{GAO}$ that will be considered along this paper, namely  $\mathbb{DGAO}$ and $\mathbb{FSGAO}$ in Section \ref{extensions}, and $\mathbb{FSDGAO}$ and $\mathbb{WGAO}$ introduced in Section \ref{one-relation}. 
 
\section{Forest frames with two relations}\label{sec:tworelations}
This section is dedicated to investigate the relational structures used in the previous results and that are made, for a GAO $({\bf A}, \Box, \Diamond)$, of the forest ${\bf F}_{\bf A}$ of its prime filters and two binary relations $R_\Box$ and $R_\Diamond$ respectively satisfying the properties (M): {\em monotonicity} in the first argument, and (A): {\em antimonotonicity} in the fist argument.

The idea of defining relational structures on the prime spectrum of the modal algebra is not new and indeed the below definition of forest frame is strongly inspired by the usual way relational structures are defined for intuitionistic modal logic and, in particular, for the logic denoted ${\bf IntK}$ in \cite{Wolter} in  which the two modalities $\Box$ and $\Diamond$ have no axioms in common and hence they are treated, on the relational side, by two (independent) accessibility relations. 

After defining forests frames, we will compare these structures with their analogous considered by Bo\v{z}i\'c and Do\v{s}en in  \cite{BoDo} and later by  Or{\l}owska and  Rewitzky in \cite{Orlo}, and by Palmigiano in \cite{Ale}.
\begin{definition}\label{def1}
A {\em forest frame} is a triple $(\bF, R_\Box, R_\Diamond)$ where $\bF=(F,\leq)$ is a finite forest and $R_\Box, R_\Diamond\subseteq F\times F$ respectively satisfy the following conditions:
\begin{itemize}
\item[$(M)$]  for all $x, y, z\in F$, if $x\leq y$ and $x R_\Box z $, then  $y R_\Box z$;
\item[$(A)$] for all $x, y, z\in F$, if $y\leq x$ and $x R_\Diamond z $, then  $y R_\Diamond z$.
\end{itemize}
\end{definition}

We have seen in the previous section that for every forest frame $(\bF, R_\Box,R_\Diamond)$ we have an associated GAO $({\bf G}({\bF}),\beta,\delta)$ and for every GAO $(\mathbf{A},\Box,\Diamond)$ we have an associated forest frame $(\bF_\mathbf{A},R_\square,R_\Diamond) $ where the relations $R_\Box , R_\Diamond\subseteq F_\mathbf{A}\times F_\mathbf{A}$ are the ones defined in the previous section.

\begin{remark}\label{remMA}
Notice that  conditions (M) and (A) in the definition above can be equivalently expressed as follows:
\begin{itemize}
\item[(M)] $(\geq \circ R_\Box) \subseteq R_\Box$ 
\item[(A)] $(\leq \circ R_\Diamond) \subseteq R_\Diamond$
\end{itemize}
where $\circ$ denotes the composition of relations. Since the converse inclusions always hold, these conditions can in turn be equivalently expressed as identities as follows:
\begin{itemize}
\item[(M)] $(\geq \circ R_\Box) = R_\Box$ 
\item[(A)] $(\leq \circ R_\Diamond) = R_\Diamond$. 
\end{itemize}
The same conditions (M) and (A) are considered by Bo\v{z}i\'c and Do\v{s}en in  \cite{BoDo} in the framework of relational models for intutionistic modal logics. More precisely, condition (M) is the one that defines in \cite{BoDo} the so-called  {\em condensed H$_\Box$ frames}, while (A) defines the {\em condensed H$_\Diamond$ frames.}
\end{remark}

Now, let  $(\bF, R_\Box, R_\Diamond)$ be any forest frame and let us define the following two binary relations on $F$:
\begin{equation}\label{eqRprime}
R'_\Box=R_\Box\circ \geq\mbox{ and }R'_\Diamond=R_\Diamond\circ \leq.
\end{equation}
In other words, for all $x,y\in F$, $x R'_\Box y$ iff there exists $z\in F$ such that $x R_\Box z $ and $z\geq y$. Analogously, $x R'_\Diamond y$ iff there exists $z\in F$ such that $x R_\Diamond z $ and $z\leq y$. Then, the following holds.
\begin{proposition}\label{prop:Rprime}
For every forest frame  $(\bF, R_\Box, R_\Diamond)$ the following conditions hold:
\begin{enumerate}
\item $(\bF, R'_\Box, R'_\Diamond)$ is a forest frame;
\item $R'_\Box(x)={\downarrow}R_\Box(x)$ and $R'_\Diamond(x)={\uparrow}R_\Diamond(x)$;
\item $(\geq\circ R_\Box'\circ \geq) = R_\Box'$ and $(\leq\circ R_\Diamond'\circ \leq) = R_\Diamond'$.
\end{enumerate}
\end{proposition}
\begin{proof}
(1) Let $x, y, z\in F$ such that $y\leq x$ and $x R_\Diamond' z$. Then, there exists $w\in F$ such that $x R_\Diamond w$ and $w\leq z$. Since $R_\Diamond$ satisfies (A), $y R_\Diamond w$. From, $y R_\Diamond w$ and $w\leq z$, we get $y R_\Diamond' z$.

Let $x, y, z\in F$ such that $x\leq y$ and $x R_\Box' z$. Then, there exists $w\in F$ such that $x R_\Box w$ and $z\leq w$. Since $R_\Box$ satisfies (M), $y R_\Box w$. From, $y R_\Box w$ and $z\leq w$, we get $y R_\Box' z$.

(2) It follows from the definition of $R_\Diamond'$ and $R_\Box'$.

(3) The inclusion $R'_\Box\subseteq (\geq\circ R'_\Box\circ \geq)$ is immediate. Let $x,y,z,w\in F$ such that $x\geq y$, $y R'_\Box z$ and $z\geq w$. We will prove that $x R'_\Box w$. Since $R'_\Box$ satisfies (M), $x R'_\Box z$. Since $R'_\Box(x)$ is an downset of $\bF$ we get that $x R'_\Box w$.

The equality $(\leq\circ R_\Diamond'\circ \leq) = R_\Diamond'$ is proved in an analogous way, 
\end{proof}

For every forest frame $(\bF, R_\Box,R_\Diamond)$, let ${\bf G}({\bF})$ be the G\"odel algebra of downsets of $\bF$  
and let the maps $\beta,\delta: G(\bF)\to G(\bF)$ be as in the previous section:
for every $a\in G(\bF)$
\begin{equation}\label{eqBetaR}
 \beta(a)=\{y\in F\mid \forall z\in F, \; (y R_\Box z\Rightarrow z\in a)\}.
 \end{equation}
\begin{equation}\label{eqMQDiamond}
\delta(a)=\{y\in F\mid \exists z\in a,\;y R_\Diamond z\}.
\end{equation}
\begin{notation}
Along this section, we will make use of subscripts to distinguish a binary relation $R$ from those that we will write $R'$, $R''$, etc. More precisely, we will use symbols $R_\Box$, $R_\Diamond$ as usual for the binary relations of a forest frame and $R'_\Box$, $R'_\Diamond$, $R''_\Box$, $R''_\Diamond$ for derived binary relations on the same forest. In addition, we will denote by $\beta'$, $\delta'$ and $\beta''$, $\delta''$ the operations on ${\bf G}({\bF})$ defined as in (\ref{eqBetaR}) and (\ref{eqMQDiamond}) by the relations $R'_\Box$, $R'_\Diamond$ and $R''_\Box$, $R''_\Diamond$ respectively.
\end{notation}
\begin{lemma} \label{diamond} Let  $(\bF, R_\Box,R_\Diamond)$ be a forest frame, then
\begin{enumerate}
\item $\delta(a)=\delta'(a)$ for all $a\in G(\bF)$.
\item $\beta(a)=\beta'(a)$ for all $a\in G(\bF)$.
\end{enumerate}

\end{lemma}

\begin{proof}
(1) We will prove $\delta'(a)\subseteq \delta(a)$. The other inclusion follows immediately. Let $y\in \delta'(a)$. Then, there exists $z\in a$ such that $y R'_\Diamond z$. From definition of $R'_\Diamond$ there exists $w\in F$ such that $y R_\Diamond w$ and $w\leq z$. Since $a\in G(\bF)$, $a$ is a downward closed subset of $\bF$ and we get that $w\in a$. Therefore $y\in \delta(a)$.

(2) We will prove $\beta(a)\subseteq \beta'(a)$. The other inclusion follows immediately. Let $y\in \beta(a)$. Let $z\in F$ such that $y R'_\Box z$, we will prove that $z\in a$. From definition of $R'_\Box$ there exists $w\in F$ such that $y R_\Box w$ and $z\leq w$. Then, $w\in a$ and since $a\in G(\bF)$, $a$ is a downward closed subset of $\bF$ and we get that $z\in a$. Therefore $y\in \beta'(a)$.
\end{proof}

From the previous lemma we get that the forest frames $(\bF, R_\Box,R_\Diamond)$ and $(\bF, R'_\Box,R'_\Diamond)$ induce the same G\"odel algebra with operators, i.e.,  $({\bf G}({\bF}),\beta,\delta)=({\bf G}({\bF}),\beta',\delta')$.  

In \cite{Orlo} Or{\l}owska and Rewitzky defined a class of relational frames, based on posets, being a dual semantics for Heyting algebras with operators. Our interest now is to compare forest frames with them. For this, we will define Or{\l}owska and Rewitzky frames on forests as follows.
\begin{definition}
An {\em Or{\l}owska-Rewitzky frame} (or {\em OR-frame} for short) is a triple $(\bF,R_\Box,R_\Diamond)$ where $\bF$ is a forest and $R_\Box,R_\Diamond\subseteq F\times F$ satisfy the following conditions:
\begin{enumerate}
\item[(OR1)] $(\geq \circ R_\Box \circ \geq) \subseteq R_\Box$,
\item[(OR2)] $(\leq \circ R_\Diamond \circ \leq) \subseteq R_\Diamond$.
\end{enumerate}
\end{definition}
As a follow-up of the above Remark \ref{remMA}, it is interesting to notice that, similarly to (M) and (A) of Definition \ref{def1}, also the aforementioned properties (OR1) and (OR2), have indeed been considered in the paper \cite{BoDo} by Bo\v{z}i\'c and Do\v{s}en. These properties, that as we will see in a while are more specific than (M) and (A) above, are those that respectively define in \cite{BoDo} the so called {\em strictly condensed H$_\Box$} and {\em stricly condensed H$_\Diamond$} frames.

The next result is a direct consequence of Proposition \ref{prop:Rprime} and Lemma \ref{diamond}.
\begin{corollary}\label{cor:HKframe}
 Let $(\bF, R_\Box, R_\Diamond)$
be a forest frame. Then, $(\bF, R'_\Box, R'_\Diamond)$ is a OR-frame. Moreover, $({\bf G}({\bF}), \beta, \delta)=({\bf G}({\bF}), \beta', \delta')$. 
\end{corollary}


Note that every OR-frame is a forest frame, but the converse is not always the case. In the following example we will show a forest frame that is not a OR-frame. 

\begin{example}\label{EXOrForest} Consider the following forest frame $(\{f_1,f_2,f_3\},\leq, R_\Box, R_\Diamond)$ where
\[R_\Box=\{(f_1, f_1), (f_2, f_3), (f_2, f_2), (f_3, f_1), (f_3, f_3)\}\]
as we see in Figure \ref{fig2} and \[R_\Diamond=\{(f_1, f_2), (f_1, f_3), (f_2, f_2), (f_2, f_1), (f_3, f_3)\}\]
as we see in Figure \ref{fig1}. If we compute $R'_\Box$ and $R'_\Diamond$ as in (\ref{eqRprime}) we get that:
\[R'_\Box=\{(f_1, f_1), (f_2, f_3), (f_2, f_2), (f_2, f_1),(f_3,f_1), (f_3, f_3)\}.\]
\begin{figure}
\begin{center}
\begin{tikzpicture}
  \node [label=below:{$f_2$}, label=below:{} ] (n1)  {} ;
  \node [right of=n1,label=below:{$f_1$}, label=below:{} ] (n2)  {} ;
  \node [above of=n2, label=above:{$f_3$}, label=below:{} ] (n3)  {} ;
  \draw [fill] (n1) circle [radius=.5mm];
    \draw [fill] (n2) circle [radius=.5mm];
      \draw [fill] (n3) circle [radius=.5mm];
   \draw (n2)--(n3);
       \draw [->,line width=0.07mm] (n1) edge[bend left] (n3);
          \draw [->,line width=0.07mm] (n3) edge[bend left] (n2);
                  \draw [->,line width=0.07mm] (n2) arc [radius=3.5mm, start angle=430, end angle= 90]  (n2);
           \draw [->,line width=0.07mm] (n1) arc [radius=3.5mm, start angle=430, end angle= 90]  (n1);
             \draw [->,line width=0.07mm] (n3) arc [radius=3.5mm, start angle=610, end angle= 270]  (n3);
\end{tikzpicture}
\hspace{.4cm}
\begin{tikzpicture}
  \node [label=below:{$f_2$}, label=below:{} ] (n1)  {} ;
  \node [right of=n1,label=below:{$f_1$}, label=below:{} ] (n2)  {} ;
  \node [above of=n2, label=above:{$f_3$}, label=below:{} ] (n3)  {} ;
  \draw [fill] (n1) circle [radius=.5mm];
    \draw [fill] (n2) circle [radius=.5mm];
      \draw [fill] (n3) circle [radius=.5mm];
   \draw (n2)--(n3);
    \draw [->,line width=0.07mm] (n1) edge[bend left] (n2);
       \draw [->,line width=0.07mm] (n1) edge[bend left] (n3);
          \draw [->,line width=0.07mm] (n3) edge[bend left] (n2);
                  \draw [->,line width=0.07mm] (n2) arc [radius=3.5mm, start angle=430, end angle= 90]  (n2);
           \draw [->,line width=0.07mm] (n1) arc [radius=3.5mm, start angle=430, end angle= 90]  (n1);
             \draw [->,line width=0.07mm] (n3) arc [radius=3.5mm, start angle=610, end angle= 270]  (n3);
\end{tikzpicture}
\caption{From left to right: The frame $(\{f_1,f_2,f_3\},\leq, R_\Box)$; the frame $(\{f_1,f_2,f_3\},\leq, R'_\Box)$.}
\label{fig2}
\end{center}
\end{figure}
and
\[R'_\Diamond=\{(f_1, f_2), (f_1, f_3), (f_2, f_2), (f_2, f_1),(f_2,f_3), (f_3, f_3)\}.\]
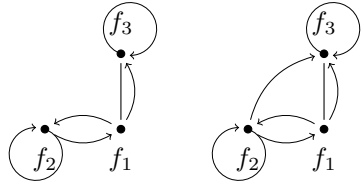
\begin{figure}
\begin{center}
\begin{tikzpicture}
  \node [label=below:{$f_2$}, label=below:{} ] (n1)  {} ;
  \node [right of=n1,label=below:{$f_1$}, label=below:{} ] (n2)  {} ;
  \node [above of=n2, label=above:{$f_3$}, label=below:{} ] (n3)  {} ;
  \draw [fill] (n1) circle [radius=.5mm];
    \draw [fill] (n2) circle [radius=.5mm];
      \draw [fill] (n3) circle [radius=.5mm];
   \draw (n2)--(n3);
    \draw [->,line width=0.07mm] (n1) edge[bend right] (n2);
          \draw [->,line width=0.07mm] (n2) edge[bend right] (n3);
                \draw [->,line width=0.07mm] (n2) edge[bend right] (n1);
           \draw [->,line width=0.07mm] (n1) arc [radius=3.5mm, start angle=430, end angle= 90]  (n1);
             \draw [->,line width=0.07mm] (n3) arc [radius=3.5mm, start angle=610, end angle= 270]  (n3);
\end{tikzpicture}
\hspace{.4cm}
\begin{tikzpicture}
  \node [label=below:{$f_2$}, label=below:{} ] (n1)  {} ;
  \node [right of=n1,label=below:{$f_1$}, label=below:{} ] (n2)  {} ;
  \node [above of=n2, label=above:{$f_3$}, label=below:{} ] (n3)  {} ;
  \draw [fill] (n1) circle [radius=.5mm];
    \draw [fill] (n2) circle [radius=.5mm];
      \draw [fill] (n3) circle [radius=.5mm];
   \draw (n2)--(n3);
    \draw [->,line width=0.07mm] (n1) edge[bend right] (n2);
       \draw [->,line width=0.07mm] (n1) edge[bend left] (n3);
          \draw [->,line width=0.07mm] (n2) edge[bend right] (n3);
                \draw [->,line width=0.07mm] (n2) edge[bend right] (n1);
           \draw [->,line width=0.07mm] (n1) arc [radius=3.5mm, start angle=430, end angle= 90]  (n1);
             \draw [->,line width=0.07mm] (n3) arc [radius=3.5mm, start angle=610, end angle= 270]  (n3);
\end{tikzpicture}
\caption{From left to right: The frame $(\{f_1,f_2,f_3\},\leq, R_\Diamond)$; the frame $(\{f_1,f_2,f_3\},\leq, R'_\Diamond)$.}
\label{fig1}
\end{center}
\end{figure}
\end{example}

Summing up, what we presented so far, shows that OR-frames form a class of relational frames strictly contained in that of forest frames. However, for every forest frame $(\bF, R_\Box, R_\Diamond)$, it is always possible to define an OR-frame $(\bF, R'_\Box, R'_\Diamond)$ based on the same forest $\bF$ such that they define the same GAO $({\bf G}({\bF}), \beta, \delta)$. 

Now, we turn our attention on the relational frames that correspond to those that satisfy the conditions of both Bo\v{z}i\'c and Do\v{s}en's $H_\Box$- and $H_\Diamond$-frames \cite{BoDo},  and then later considered by Palmigiano in \cite{Ale}. Again, we will consider the particular case of relational structures based on forests, rather than the more general case studied in \cite{Ale}.

\begin{definition}
A {\em Palmigiano frame} (or {\em P-frame} for short), is a triple $(\bF, R_\Box, R_\Diamond)$ where $\bF$ is a forest and $R_\Box, R_\Diamond\subseteq F\times F$ satisfy the following conditions:
\begin{enumerate}
\item[(P1)] $(\geq \circ R_\Box) \subseteq (R_\Box \circ \geq)$;
\item[(P2)] $(\leq \circ R_\Diamond) \subseteq (R_\Diamond \circ \leq)$.
\end{enumerate}
\end{definition}
Our next result shows that P-frames include forest frames.

\begin{proposition} Every forest frame $(\bF,R_\Box,R_\Diamond)$is a P-frame. 
\end{proposition}
\begin{proof}
Let $(\bF,R_\Box,R_\Diamond)$ be a forest frame. It is easy to see that $(\geq\circ R_\Box)\subseteq R_\Box \subseteq (R_\Box\circ \geq
)$ and $(\leq \circ R_\Diamond)\subseteq R_\Diamond\subseteq (R_\Diamond\circ \leq)$ and the result follows. 
\end{proof}

So, in particular we have that every OR-frame is a forest frame and every forest frame is a P-frame  over the same ordered set. So we have an inclusion of frame classes.

Now, given a P-frame $(\bF, R_\Box, R_\Diamond)$, consider the following relations $R_\Box'',R_\Diamond''\subseteq F\times F$ defined by:
\begin{equation}\label{eqRsecond}
R_\Box''=(\geq \circ R_\Box)\text{ and }R_\Diamond''=(\leq \circ R_\Diamond).
\end{equation}

\begin{proposition} Let $(\bF, R_\Box, R_\Diamond)$ be a P-frame. Then, $(\bF,R_\Box'',R_\Diamond'')$ is a forest frame such that
\begin{enumerate}
\item $\delta(a)=\delta''(a)$ for all $a\in G(\bF)$.
\item $\beta(a)=\beta''(a)$ for all $a\in G(\bF)$.
\end{enumerate}
\end{proposition}

\begin{proof}
It is immediate to see that $(\bF,R_\Box'',R_\Diamond'')$ is a forest frame.

1. Let $a\in G(\bF)$. Since $R_\Diamond\subseteq R_\Diamond''$, $\delta(a)\subseteq \delta''(a)$. Let $x\in \delta''(a)$. Then, $R_\Diamond''(x)\cap a\neq \emptyset$. Let $y\in R_\Diamond''(x)$ such that $y\in a$. Since $(\bF, R_\Box, R_\Diamond)$ is a P-frame we have that $R_\Diamond''\subseteq (R_\Diamond\circ \leq)$. So, there exists $z\in F$ such that $x R_\Diamond z $ and $z\leq y$. Thus, since $a$ is a downset, $z\in a$. Therefore $R_\Diamond(x)\cap a\neq \emptyset$ and $\delta''(a)\subseteq \delta(a)$.

2. Let $a\in G(\bF)$. Since $R_\Box\subseteq R_\Box''$, $\beta''(a)\subseteq \beta(a)$. Let $x\in \beta(a)$. Then, $R_\Box(x)\subseteq a$. Let $y\in R_\Box''(x)$. Since $(\bF, R_\Box, R_\Diamond)$ is a P-frame, we have that $R_\Box''\subseteq (R_\Box\circ \geq)$. Therefore, there exists $z\in F$ such that $x R_\Box z $ and $y\leq z$. Thus, $z\in a$ and since $a$ is a downset, $y\in a$. Therefore $R_\Box''(x)\subseteq a$ and $\beta(a)\subseteq \beta''(a)$.
\end{proof}
The next example shows that forest frames form a proper subclass of P-frames. Thus, together with the above results and Example \ref{EXOrForest}, we have that OR-frames are strictly contained in the class of forest frames that, in turn, are strictly contained in P-frames.
\begin{example}\label{EXPForest}
Consider the  forest $\bF$ (tree) depicted as in Figure \ref{FigTree}.
 \begin{figure}[h]
 \begin{center}
 \begin{tikzpicture}
\node [label=right:{${x}$}, label=below:{ }] (n1)  {} ;
\node [above  of=n1,label=right:{${y}$}, label=below:{ }] (n2)  {} ;
\node [above right of=n2,label=right:{${k}$}, label=below:{ }] (n3)  {} ;
\node [above left of=n2,label=left:{${z}$}, label=below:{ }] (n4)  {} ;

  \draw  (n1) -- (n2);
  \draw  (n2) -- (n3);
    \draw  (n2) -- (n4);
     
    \draw [fill] (n1) circle [radius=.5mm];
  \draw [fill] (n2) circle [radius=.5mm];
  \draw [fill] (n3) circle [radius=.5mm];
  \draw [fill] (n4) circle [radius=.5mm];
   \end{tikzpicture}
 \end{center}
 \caption{A tree with 4 points used in Example \ref{EXPForest}.}\label{FigTree}
 \end{figure}
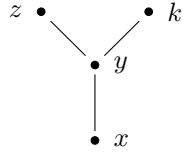
  Consider the following relations  $R_\Box$ and $R_\Diamond$ on $F$:
$$
R_\Box=\{(x,y),(y,z),(z,z),(k,z)\}
$$
and 
$$
R_\Diamond=\{(x,x),(y,y)\}
$$
Now, we compute the following composed  relations: 
$$
(\geq \circ R_\Box)=\{(x,y),(y,y),(y,z),(z,y),(z,z),(k,y),(k,z)\}
$$
and 
$$
(R_\Box\circ \geq) =\{(x,x),(x,y),(y,x),(y,y),(y,z),(z,x),(z,y),(z,z),(k,x),(k,y),(k,z)\}.
$$
Moreover, we have: 
$$
(\leq \circ R_\Diamond)=\{(x,x),(x,y),(y,y)\}
$$
and 
$$
(R_\Diamond\circ \leq) =\{(x,x),(x,y),(x,z),(x,k),(y,y),(y,z),(y,k)\}.
$$
This shows that $(\bF,R_\Box, R_\Diamond)$ is a P-frame but it  is not a forest frame.
\end{example}
Now, we end this section with a (graphical) comparison between the characterizing properties for $R_\Box$ and $R_\Diamond$ of OR-frames, forest-frames and P-frames. 
\begin{remark}
So far, we have considered three kind of relational structures based on forests, namely forest frames, OR-frames, and P-frames. For the next comparison, let us recall what properties are asked for the binary relations $R_\Box$ an $R_\Diamond$ in each of the aforementioned models.

As for forest frames, we have the following two properties to be satisfied by $R_\Box$ and $R_\Diamond$ respectively.
\begin{itemize}
\item[$(M)$]  for all $x, y, z\in F$, if $x\leq y$ and $x R_\Box z $, then  $y R_\Box z$;
\item[$(A)$] for all $x, y, z\in F$, if $y\leq x$ and $x R_\Diamond z $, then  $y R_\Diamond z$.
\end{itemize}
As for OR-frames and P-frames, let us express (OR1), (OR2), (P1) and (P2) by first-order formulas as follows:
\begin{itemize}
\item[(OR1)] for all $x,y\in F$, if there exist $z,w\in F$ such that $x\geq z$, $z R_\Box w$ and $w\geq y$, then $x R_\Box y$. 
\item[(OR2)] for all $x,y\in F$, if there exist $z,w\in F$ such that $x\leq z$, $z R_\Diamond w$ and $w\leq y$, then $x R_\Diamond y$. 
\item[(P1)] for all $x,y\in F$, if there exists $z\in F$ such that $x\geq z$ and $z R_\Box y$, then there exists $w\in F$ such that $x R_\Box w$ and $w\geq y$.
\item[(P2)] for all $x,y\in F$, if there exists $z\in F$ such that $x\leq z$ and $z R_\Diamond y$, then there exists $w\in F$ such that $x R_\Diamond w$ and $w\leq y$.
\end{itemize}
Figures \ref{Fig:RelationsBox} and \ref{Fig:RelationsDiamond} below present a graphical representation for the properties recalled above. The dashed arrows will represent the final relation corresponding to the right-hand side of the above quasi-equations. Also, for a point $a$, we will label it by $\exists a$ (instead of simply $a$) to highlight that the existence of $a$ is ensured by the left-hand side of the quasi-equations above. In particular, this is the case of (P1) and (P2). 
\begin{figure}[h]
\begin{center}
\begin{tikzpicture}
  \node [label=below:{$z$}, label=below:{} ] (n1)  {} ;
  \node [left of=n1,label=below:{$x$}, label=below:{} ] (n2)  {} ;
  \node [above of=n2, label=above:{$y$}, label=below:{} ] (n3)  {} ;
  \draw [fill] (n1) circle [radius=.5mm];
    \draw [fill] (n2) circle [radius=.5mm];
      \draw [fill] (n3) circle [radius=.5mm];
   \draw (n2)--(n3);
       \draw [->] (n2) edge[] (n1);
          \draw [->,dashed] (n3) edge[] (n1);
\end{tikzpicture}
\hspace{2cm}
\begin{tikzpicture}
  \node [label=below:{$y$}, label=below:{} ] (n1)  {} ;
  \node [above of=n1, label=above:{$w$}, label=below:{} ] (n4)  {} ;
  \node [left of=n1,label=below:{$z$}, label=below:{} ] (n2)  {} ;
  \node [above of=n2, label=above:{$x$}, label=below:{} ] (n3)  {} ;
  \draw [fill] (n1) circle [radius=.5mm];
    \draw [fill] (n2) circle [radius=.5mm];
      \draw [fill] (n3) circle [radius=.5mm];
      \draw [fill] (n4) circle [radius=.5mm];
   \draw (n2)--(n3);
   \draw (n1)--(n4);
    \draw [->,dashed] (n3) edge[] (n1);
       \draw [->] (n2) edge[] (n4);
\end{tikzpicture}
\hspace{2cm}
\begin{tikzpicture}
  \node [label=below:{$y$}, label=below:{} ] (n1)  {} ;
  \node [above of=n1, label=above:{$\exists w$}, label=below:{} ] (n4)  {} ;
  \node [left of=n1,label=below:{$z$}, label=below:{} ] (n2)  {} ;
  \node [above of=n2, label=above:{$x$}, label=below:{} ] (n3)  {} ;
  \draw [fill] (n1) circle [radius=.5mm];
    \draw [fill] (n2) circle [radius=.5mm];
      \draw [fill] (n3) circle [radius=.5mm];
      \draw [fill] (n4) circle [radius=.5mm];
   \draw (n2)--(n3);
   \draw (n1)--(n4);
    \draw [->] (n2) edge[] (n1);
          \draw [->,dashed] (n3) edge[] (n4);
\end{tikzpicture}

\end{center}
\caption{A graphical representation of the properties (M) (left-hand side), (OR1)  (center) and (P1) (right-hand side).}\label{Fig:RelationsBox}
\end{figure}
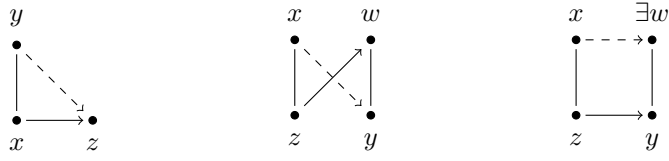

\begin{figure}[h]
\begin{center}
\begin{tikzpicture}
  \node [label=above:{$z$}, label=below:{} ] (n1)  {} ;
  \node [left of=n1,label=above:{$x$}, label=below:{} ] (n2)  {} ;
  \node [below of=n2, label=below:{$y$}, label=below:{} ] (n3)  {} ;
  \draw [fill] (n1) circle [radius=.5mm];
    \draw [fill] (n2) circle [radius=.5mm];
      \draw [fill] (n3) circle [radius=.5mm];
   \draw (n2)--(n3);
       \draw [->] (n2) edge[] (n1);
          \draw [->,dashed] (n3) edge[] (n1);
\end{tikzpicture}
\hspace{2cm}
\begin{tikzpicture}
  \node [label=below:{$w$}, label=below:{} ] (n1)  {} ;
  \node [above of=n1, label=above:{$y$}, label=below:{} ] (n4)  {} ;
  \node [left of=n1,label=below:{$x$}, label=below:{} ] (n2)  {} ;
  \node [above of=n2, label=above:{$z$}, label=below:{} ] (n3)  {} ;
  \draw [fill] (n1) circle [radius=.5mm];
    \draw [fill] (n2) circle [radius=.5mm];
      \draw [fill] (n3) circle [radius=.5mm];
      \draw [fill] (n4) circle [radius=.5mm];
   \draw (n2)--(n3);
   \draw (n1)--(n4);
    \draw [->,dashed] (n2) edge[] (n4);
       \draw [->] (n3) edge[] (n1);
\end{tikzpicture}
\hspace{2cm}
\begin{tikzpicture}
  \node [label=below:{$\exists w$}, label=below:{} ] (n1)  {} ;
  \node [above of=n1, label=above:{$y$}, label=below:{} ] (n4)  {} ;
  \node [left of=n1,label=below:{$x$}, label=below:{} ] (n2)  {} ;
  \node [above of=n2, label=above:{$z$}, label=below:{} ] (n3)  {} ;
  \draw [fill] (n1) circle [radius=.5mm];
    \draw [fill] (n2) circle [radius=.5mm];
      \draw [fill] (n3) circle [radius=.5mm];
      \draw [fill] (n4) circle [radius=.5mm];
   \draw (n2)--(n3);
   \draw (n1)--(n4);
    \draw [->,dashed] (n2) edge[] (n1);
          \draw [->] (n3) edge[] (n4);
\end{tikzpicture}

\end{center}
\caption{A graphical representation of the properties (A) (left-hand side), (OR2)  (center) and (P2) (right-hand side).}\label{Fig:RelationsDiamond}
\end{figure}
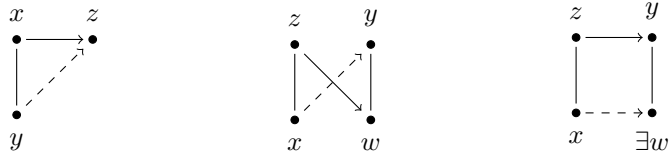

\end{remark}

\section{Adding structure to G\"odel algebras with operators}\label{extensions}

 In this section we will be concerned with two  extensions of G\"odel algebras with operators and their forest frame semantics. The first one is the subvariety $\mathbb{DGAO}$ of $\mathbb{GAO}$ obtained by the equations (D1) and (D2)  below:
 \begin{itemize}
\item[$(D1)$] $\Box(a\vee b)\leq \Box a\vee \Diamond b$;
\item[$(D2)$] $\Box a\wedge \Diamond b\leq \Diamond (a\wedge b)$.
\end{itemize}
The above axioms  have been firstly studied by Dunn in \cite{Dunn} in the logical setting of the positive fragment of classical (or intuitionistic) logic. The algebras in $\mathbb{DGAO}$ will be henceforth called {\em Dunn} GAOs. 

Before moving to the second variety, let us show an interesting property of Dunn GAOs, the fact that the modal operators are closed on the set of Boolean elements algebras in $\mathbb{DGAO}$.
Let $(\mathbf{A},\Box,\Diamond)$ be a GAO and let us consider the set of Boolean elements of $\mathbf{A}$:
\[B(\mathbf{A})=\{x\in A: x\vee\neg x=\top\}.\] Then the following holds.
\begin{proposition} \label{bool}
Let $(\mathbf{A},\Box,\Diamond)$ be a Dunn GAO, then $\Box a,\Diamond a\in B(\mathbf{A})$ for all $a\in B(\mathbf{A})$.
\end{proposition}
\begin{proof}
Let $a\in B(\mathbf{A})$. Then, $\top=\Box(a\vee \neg a)\leq \Box a\vee \Diamond (\neg a)$. Also, $\Box a\wedge \Diamond (\neg a)\leq \Diamond (a\wedge \neg a)=\bot$. So, we get that $\Diamond (\neg a)\leq \neg \Box a$. Therefore, $\top=\Box a\vee \Diamond (\neg a)\leq \Box a\vee \neg \Box a$. And thus $\Box a\in B(\mathbf{A})$.

On the other hand, let $a\in B(\mathbf{A})$. Then, $\top=\Box(\neg a\vee a)\leq \Box(\neg a)\vee \Diamond a$. Also, $\Box(\neg a)\wedge \Diamond a \leq \Diamond( a\wedge \neg a)=\bot$. So, we get that $\Box (\neg a)\leq \neg \Diamond a$. Therefore, $\top=\Box (\neg a)\vee \Diamond a\leq \neg \Diamond a\vee \Diamond a$. And thus $\Diamond a\in B(\mathbf{A})$.
\end{proof}

The second is the variety $\mathbb{FSGAO}$ given by the well-known Fischer Servi axioms (FS1) and (FS2):

 \begin{itemize}
\item[$(FS1)$] $\Diamond (a\to b)\leq (\Box a\to \Diamond b)$;
\item[$(FS2)$] $(\Diamond a\to \Box b)\leq \Box (a\to b)$.
\end{itemize}
Algebras in $\mathbb{FSGAO}$ will be called {\em Fischer Servi} GAOs.
 
In  \cite{Palmigiano2}  it has been  proved that (D2) and (FS1) are  equivalent. For the sake of completeness, we provide another proof of this fact in the result below.
\begin{proposition}\label{prop:D2-FS2}
Every GAO $({\bf A}, \Box,\Diamond)$ satisfies (D2) iff it satisfies (FS1). 
\end{proposition}
\begin{proof}
Let us start assuming that $({\bf A}, \Box,\Diamond)$ satisfies (D2), that is, for all $a, b\in A$, $\Box a\wedge \Diamond b\leq \Box(a\wedge b)$. Let us show that $({\bf A}, \Box,\Diamond)$ thus satisfies (FS2): $\Diamond (a\to b)\leq \Box a\to \Diamond b$. The latter, by residuation, is equivalent to $\Diamond (a\to b)\wedge \Box a\leq  \Diamond b$. Now, by (D2) $\Box a\wedge \Diamond (a\to b)\leq \Diamond(a\wedge (a\to b))$ and the latter equals $\Diamond (a\wedge b)\leq \Diamond b$. 

Conversely, let us assume that $({\bf A}, \Box,\Diamond)$ satisfies (FS1), that is, for all $a, b \in A$, $\Diamond (a\to b)\leq \Box a\to \Diamond b$. By residuation, $\Diamond b\leq \Diamond (a\to (a\wedge b))$ and by (FS1), $\Diamond (a\to (a\wedge b))\leq \Box a \to \Diamond (a\wedge b)$. Thus, we conclude that $\Diamond b\wedge \Box a\leq \Diamond (a\wedge b)$. 
\end{proof}

A full comparison between Dunn's and Fischer Servi's axioms, at best of our knowledge, has not be presented. This section is hence dedicated to a comparison between these two axiom schema.
 
 By Corollary \ref{corollarySubclass}, it is clear that each GAO in either $\mathbb{DGAO}$ or $\mathbb{FSGAO}$ has an isomorphic representation, within the same classes $\mathbb{DGAO}$ and $\mathbb{FSGAO}$ respectively, in the sense of Theorem \ref{thm:isoBox}. However, in these particular cases, it is possible to consider special forest frames with only one accessibility relation that, equivalently to the forest frames studied in Subsection \ref{sec:tworelations}, allows to recover, up to isomorphism, the G\"odel algebra with operators we started with. Let us hence introduce the following.
 
 \begin{definition} \label{basic}
 A {\em basic-frame} is a pair $(\bF, R)$ where $\bF$ is a finite forest, $R\subseteq F\times F$ and  there exist $R_\Box, R_\Diamond\subseteq F\times F$ such that:
 \begin{itemize}
 \item[(1)] $R_\Box$ satisfies (M) and $R_\Diamond$ satisfies (A);
 \item[(2)] $R=R_\Box\cap R_\Diamond$.
 \end{itemize}
 \end{definition}
  Clearly, every forest frame defines a basic-frame by taking $R=R_\Box\cap R_\Diamond$ and, vice-versa, if $(\bF, R)$ is a basic-frame where  $R=R_\Box\cap R_\Diamond$, then $(\bF, R_\Box, R_\Diamond)$ is a forest frame. 
 
 In order for our next result to be clear, let us introduce the following notation. For any GAO $({\bf A},\Box, \Diamond)$ be a GAO, in accordance with the notation used in the previous section, we will denote by $\bF_{\bf A}$ the forest of its prime filters, and by $R_\Box$ and $R_\Diamond$ the binary relations on $F_{\bf A}$ defined as in (\ref{eq:QBox}) and (\ref{eq:Qdiamond}) respectively.  
 Then, by Proposition \ref{prop:Rprime} (1), it follows that $(\bF_{\bf A}, R)$, where  $R_{\bf A}=R_\Box\cap R_\Diamond$, is a basic frame.
  %
Then, let $\beta$ and $\delta$ be the unary operators on ${\bf G}(\bF_{\bf A})$ defined as in \eqref{eqBetaR} and \eqref{eqMQDiamond} resp., 
 while $\beta_{R_{\bf A}}$ and $\delta_{R_{\bf A}}$ will denote the unary maps on ${\bf G}(\bF_{\bf A})$ also defined as in (\ref{eqBetaR}) and (\ref{eqMQDiamond}) resp., but where $R_{\bf A}$ replaces both $R_\Box$ and $R_\Diamond$.  respectively. 

 As the following result shows, basic and forest frames are {\em equivalent} in the sense that they define the same Dunn GAO in the isomorphic representation theorem.
 \begin{theorem}\label{isoDUNN}
 Let $({\bf A}, \Box, \Diamond)$ be any Dunn GAO. Then 
 $$
 ({\bf A}, \Box, \Diamond)\cong ({\bf G}(\bF_{\bf A}), \beta, \delta)\cong ({\bf G}(\bF_{\bf A}), \beta_{R_{\bf A}}, \delta_{R_{\bf A}})
 $$ 
 via the same isomorphism $r$. In particular, for all $a\in A$, 
 $$
 r(\Box a)=\beta(r(a))=\beta_{R_{\bf A}}(r(a))
 $$
 and 
 $$
  r(\Diamond a)=\delta(r(a))=\delta_{R_{\bf A}}(r(a)). 
 $$
 \end{theorem}
 \begin{proof}
Let $({\bf A}, \Box, \Diamond)$ be a Dunn GAO, and let us denote by ${\bf A}^-$ its $\{\to, \neg\}$-free reduct. Then, since $({\bf A}, \Diamond, \Box)$ satisfies $(D1)$ and $(D2)$, $({\bf A}^-, \Diamond, \Box)$ is a {\em positive modal algebra} in the sense of \cite{Dunn,CelJan}. Since the set of prime filters of ${\bf A}$ and that of ${\bf A}^-$ coincide,  ${\bf F}_{{\bf A}^-}={\bf F}_{\bf A}$ and, following \cite{CelJan}, let us define $R_{\bf A}\subseteq F_{{\bf A}^-}\times F_{{\bf A}^-}$  as follows: for all $f_1, f_2\in F_{\bf A}$,
\begin{center}
$R_{\bf A}(f_1, f_2)$ iff $\Box^{-1}(f_1)\subseteq f_2\subseteq \Diamond^{-1}(f_1)$.
\end{center}
Since $f_2\subseteq \Diamond^{-1}(f_1)$ iff $\Diamond(f_2)\subseteq f_1$, by \cite[Lemma 2.1(1)]{CelJan}, we have that $R_{\bf A}=R_\Box\cap R_\Diamond$, where $R_\Box$ and $R_\Diamond$ are defined as usual.

Now, let ${\bf S}_{{\bf F}_{{\bf A}^-}}$ be the G\"odel algebra of subforests of ${\bf F}_{{\bf A}^-}$ and define $\delta_{R_{\bf A}}$ and $\beta_{R_{\bf A}}$ on $S_{{\bf F}_{{\bf A}^-}}$ by (\ref{eqMQDiamond}) and (\ref{eqBetaR}) respectively. Then, \cite[Theorem 2.2]{CelJan} (see also \cite[Theorem 8.1]{Dunn}), shows that $({\bf A}^-, \Diamond, \Box)$ and the positive algebra $(({\bf S}_{{\bf F}_{{\bf A}^-}})^-, \delta_{R_{\bf A}}, \beta_{R_{\bf A}})$ are isomorphic (as positive modal algebras).

Since ${\bf F}_{\bf A}={\bf F}_{{\bf A}^-}$, one has that ${\bf S}_{{\bf F}_{{\bf A}^-}}={\bf G}(\bF_{\bf A})$. Now, it is not difficult to extend the above result to Dunn GAOs   by expanding the positive modal algebra $(({\bf S}_{{\bf F}_{{\bf A}}})^-, \delta_{R_{\bf A}}, \beta_{R_{\bf A}})$ by the operator $\to$  defined as in Section \ref{sec:algandFortests}: for all $x, y\in S_{{\bf F}_{\bf A}}$, 
$$
x\to y= ({\uparrow}(x\setminus y))^c = F_{\bf A}\setminus {\uparrow}(x\setminus y).
$$ 
Then,  $({\bf S}_{{\bf F}_{{\bf A}}})^-$ plus $\to$ and $\neg$ (defined as usual by $\neg x=x\to \emptyset$) is a G\"odel algebra isomorphic to ${\bf G}(\bF_{\bf A})$.
 \end{proof}
 The case of Fischer Servi GAOs is similar and indeed the same basic frames are enough to construct the isomorphic copy of any algebra belonging to $\mathbb{FSGAO}$. The unique necessary modification consists in considering, in place of $\beta_{R_{\bf A}}$, the map $\beta_{(\geq\circ R_{\bf A})}:{\bf G}(\bF_{\bf A})\to {\bf G}(\bF_{\bf A})$ as in (\ref{eqBetaR}), but adopting the composed relation ${\geq}{\circ}R_{\bf A}$ instead of $R_\Box$.

 In order to prove that basic frames allows us to show a forest-based representation result for Fischer Servi GAOs, let us recall \cite[Theorem 4.1]{Orlo} that has been proved in the more general setting of Heyting algebras with operators that satisfy Fischer Servi equations. In the aforementioned paper, these latter algebras have been called {\em HK1-algebras}. 

\begin{theorem} [Theorem 4.1 \cite{Orlo}]
Every HK1-algebra ${\bf W}$ is embeddable into the complex algebra of its canonical frame $C(X(W))$ through the mapping $h: W \to C(X(W))$ defined as $h(w) = \{ F \in X(W) : w \in F\}$. 
\end{theorem}

Now, as we recalled above, HK1-algebras are the Heyting analogues of our Fisher-Servi GAOs, so every  $({\bf A}, \Box, \Diamond)\in \mathbb{FSGAO}$ is a HK1-algebra in particular. Moreover, the canonical frame of $({\bf A}, \Box, \Diamond)$ is exactly the basic frame $({\bf F_A}, R_{\bf A})$, where ${\bf F_A}$ is the forest of its prime filters and $R_{\bf A}$ is the binary relation on $F_{\bf A}$ defined as usual. Finally, still following \cite{Orlo}, it is immediate to see that the complex algebra of $({\bf F_A}, R_{\bf A})$ is exactly the Fischer Servi GAO $(\bf {\bf G}(\bF_{\bf A}), \beta_{\cgeq \circ R_{\bf A}}, \delta_{R_{\bf A}})$ and the mapping $h$ is the same as the mapping $r$. 

Therefore, according to the above theorem, every finite Fischer Servi GAO $({\bf A}, \Box, \Diamond)$ can be embedded (as HK1-algebra) into $({\bf G}(\bF_{\bf A}), \beta_{\cgeq \circ R_{\bf A}}, \delta_{R_{\bf A}})$ by means of $r$. Now, since $({\bf A}, \Box, \Diamond)$ is indeed a Fischer Servi  GAO, then $({\bf G}(\bF_{\bf A}), \beta_{\cgeq \circ R_A}, \delta_{R_A})$ must be a GAO as well, and since $\bf A$ is finite, according to Lemma \ref{lemma1}, $r$ is an isomorphism. Therefore we have the following representation theorem. 

 \begin{theorem}
 Let $({\bf A}, \Box, \Diamond)$ be any finite Fischer Servi GAO. Then 
 $$
 ({\bf A}, \Box, \Diamond)\cong  ({\bf G}(\bF_{\bf A}), \beta_{(\cgeq\circ R_A)}, \delta_{R_A})
 $$ 
 via the isomorphism $r$. 
\end{theorem}

The next result, 
the statement of which can be found in two papers by Celani (see \cite[Teorema 2.6]{Cel01} and \cite[Theorem 7]{Cel06}), is meant to show what is the effect on the relations of the associated frames, of Dunn and Fischer Servi equations once added to G\"odel algebras with operators. For the sake of completeness, we present a proof of it.

\begin{theorem} \label{CJteo} Let $({\bf A}, \Box, \Diamond)$ be a GAO, let $({\bf F_A}, R_\Box, R_\Diamond)$ be the dual frame, and let $R_A = R_\Box \cap R_\Diamond$. Then the following conditions hold:
 \begin{enumerate}
 \item $({\bf A}, \Box, \Diamond)$ 
 satisfies (D1) iff $R_\Box = R_A {\circ \geq}$.
 \item $({\bf A}, \Box, \Diamond)$ satisfies (D2) iff $R_\Diamond = R_A {\circ \leq}$.
 \item $({\bf A}, \Box, \Diamond)$ satisfies (FS2) iff $R_\Box = {\geq \circ} R_A $.
 \end{enumerate}
 \end{theorem}
 \begin{proof} The claims concerning left-to-right implications in (2) and (3) have been proved in \cite[Lemma 5.4]{Ale}. In particular, part of claim (2) is proved by \cite[Lemma 5.4 (2)]{Ale} together with  Proposition \ref{prop:D2-FS2} above concerning the equivalence between (D2) and (FS1). 
 
 Concerning (1), the claim has been stated in  \cite{CelJan} without proof. We are hence going to show it here. 
  \vspace{.1cm}
 
 (Left-to-right). Let us start observing that the inclusion $R_\Box \supseteq R_A {\circ \geq}$ is straightforward. Let us hence prove that $R_\Box \subseteq R_A {\circ \geq}$.
 
 Let $f_1 R_\Box f_2$ and let us prove that there exists $f_3\in F_{\bf A}$ such that 
 \begin{center}
 $\Box^{-1}(f_1)\subseteq f_3\subseteq \Diamond ^{-1}(f_1)$ and $f_2\leq f_3$. 
 \end{center}
 Let $i=Id((f_2)^c\cup (\Diamond^{-1}(f_1))^c)$ be the  ideal generated by  $(f_2)^c\cup (\Diamond^{-1}(f_1))^c$ and let us prove that $\Box^{-1}(f_1)\cap i=\emptyset$. By way of contradiction, assume that $a\in \Box^{-1}(f_1)\cap i$. Thus, in particular, $a\in \Box^{-1}(f_1)$, that is to say, $\Box a\in f_1$. Moreover, $a\in i$, whence there exists $c\in (f_2)^c$ and $b\in (\Diamond^{-1}(f_1))^c$ such that $a\leq c\vee b$. Therefore, by the monotonicity of $\Box$ and (D1), $\Box a\leq \Box(c\vee b)\leq \Box c\vee \Diamond b$. Now, since $\Box a\in f_1$ and $f_1$ is a filter, $\Box c\vee \Diamond b\in f_1$ as well. Notice that $\Box c\not\in f_1$. Indeed, since $c\not\in f_2$ and $f_2\supseteq \Box^{-1}(f_1)$, $\Box c\not\in f_1$. However, $f_1$ is prime and then $\Diamond b\in f_1$ that is absurd. Thus, $\Box^{-1}(f_1)\cap i=\emptyset$.
 
By Birkhoff prime filter theorem, there exists a prime filter $g$ such that $\Box^{-1}(f_1)\subseteq g$ and $g\cap i=\emptyset$. Therefore, $g\subseteq f_2$, meaning that in ${\bF}_{\bf A}$, $f_2\leq g$,  and $g\subseteq \Diamond^{-1}(f_1)$. Thus the claim is completed by taking $g=f_3$. 
 \vspace{.1cm}
 
 (Right-to-left). Let us assume, by way of contradiction, that $R_\Box=R_{\bf A}{\circ \geq}$ and let $a, b\in A$ such that $\Box(a\vee b)\not\leq \Box a\vee \Diamond b$. Then, there exists a prime filter $f$ such that $\Box(a\vee b)\in f$ and $\Box a\vee \Diamond b\not\in f$. Then, $\Box(a\vee b)\in f$ implies that $a\vee b\in \Box^{-1}(f)$, while $\Box a\vee \Diamond b\not\in f$ entails that, in particular, 
 \begin{equation}\label{eqContradiction}
 a\not\in \Box^{-1}(f)\mbox{ and }b\not\in \Diamond^{-1}(f).
 \end{equation}
  Since $\Box^{-1}(f)$ is a filter and $a\not\in \Box^{-1}(f)$, there is a prime filter $g$ such that $g\supseteq \Box^{-1}(f)$ and $a\not\in g$. In particular, $g\supseteq \Box^{-1}(f)$ implies that $f R_\Box g $. By  hypothesis $R_\Box=R_{\bf A}{\circ \geq}$, and since $f R_\Box g $, one has that there exists a prime filter $h$ such that $\Box^{-1}(f)\subseteq h\subseteq \Diamond^{-1}(f)$ and $h\subseteq g$. Since $a\not\in g$, and $h\subseteq g$, $a\not\in h$. However, $a\vee b\in \Box^{-1}(f)\subseteq h$, whence $b\in h\subseteq \Diamond^{-1}(f)$  contradicting  (\ref{eqContradiction}).
  
Now, in order to prove the right-to-left  claim (2), let us assume that $R_\Diamond=R_A\circ {\leq}$ and let $a,b\in A$. We want to prove that $\Diamond (a \rightarrow b)\leq \Box a\rightarrow \Diamond b$
%
or equivalently  $\Diamond (a \rightarrow b) \land \Box a \leq \Diamond b$.  
So, let us consider the filter $f={\uparrow}(\Diamond (a \rightarrow b) \land \Box a)$, and  we will prove $\Diamond b\in f$. Suppose $\Diamond b \not\in f$. Then, there exists a prime filter $f_1$ such that $f \subseteq f_1$ and $\Diamond b\notin f_1$. Then, $a\rightarrow b\in \Diamond^{-1}(f_1)$ and $a\in \Box^{-1}(f_1)$. By Lemmas \cite[3.5 and (a) of 3.2]{Ale}, there exists a prime filter $f_2$ such that $a \to b\in f_2\subseteq \Diamond^{-1}(f_1)$. 
So, $f_1 R_\Diamond f_2$ and, by assumption, there exists a prime filter $f_3$ such that $\Box^{-1}(f_1)\subseteq f_3\subseteq \Diamond^{-1}(f_1)$ and $f_2\subseteq f_3$. Since $a\rightarrow b\in f_2\subseteq f_3$ and $a\in \Box^{-1}(f_1)\subseteq f_3$, we get that $a,a\rightarrow b\in f_3$ and therefore $b\in f_3$. From $b\in f_3\subseteq \Diamond^{-1}(f_1)$ we obtain $\Diamond b\in f_1$ which is a contradiction. 
Therefore, $\Diamond b \in f$.

To prove  the right-to-left claim (3), let us assume that $R_\Box={\geq} \circ R_A$ and let $a,b\in A$. We will prove that $\Diamond a \rightarrow \Box b\leq \Box (a\rightarrow b)$. So, let us consider a prime filter $f_1$ such that $\Diamond a \rightarrow \Box b\in f_1$ and suppose that $\Box(a\rightarrow b)\notin f_1$. Since $\Box^{-1}(f_1)$ is a filter, $b\notin Fi(\Box^{-1}(f_1)\cup\{a\})$, where $Fi(\Box^{-1}(f_1)\cup\{a\})$ is the filter generated by $\Box^{-1}(f_1)\cup\{a\}$.  Then, there exists a prime filter $f_2$ such that $\Box^{-1}(f_1)\subseteq f_2$, $a\in f_2$ and $b\notin f_2$. By assumption, since $f_1 R_\Box f_2$, there exists a prime filter $f_3$ such that $f_1\subseteq f_3$ and $\Box^{-1}(f_3)\subseteq f_2\subseteq \Diamond^{-1}(f_3)$. Thus, from $\Diamond a \rightarrow \Box b\in f_1\subseteq f_3$ and $a\in f_2\subset \Diamond^{-1}(f_3)$ we get $\Diamond a \rightarrow \Box b,\Diamond a\in f_3$. We can imply that $\Box b\in f_3$ and hence $b\in \Box^{-1}(f_3)\subseteq f_2$ which is a contradiction because $b\notin f_2$. Therefore $b\in Fi(\Box^{-1}(f_1)\cup\{a\})$, and we get $a\rightarrow b\in \Box^{-1}(f_1)$.
 \end{proof}

 
As we already recalled in Section \ref{sec:algandFortests}, the algebraic category finite of G\"odel algebras are dual to the category of finite forests. The following result, whose proof can be found in \cite[\S4.2]{ABG}, hence complements Lemma \ref{lemma1} and it recalls that every finite forest $({\bf X},\leq)$ is isomorphic to the finite forest of prime filters of the G\"odel algebra ${\bf G}({\bf X})$.
 
 \begin{lemma} [C.f. \cite{ABG}] \label{bijective} Let $(X, \leq)$ be a finite forest and let $F_{G(X)}$ be the set of prime filters of ${\bf G(X)}$.  Define the mapping $k: X \to F_{G(X)}$ as follows: for any $x \in X$, by $k(x) = \{f \in G(X) \mid x \in f\}$. Then $k$ is a bijective mapping such that $x \leq y$ iff $k(x) \supseteq k(y)$. Therefore, the forests ${\bf X} = (X, \leq)$ and ${\bf F_{G(X)}} = (F_{G(X)}, \supseteq)$ are isomorphic through the mapping $k$. 
\end{lemma}

The mapping $k$ appearing in the above lemma preserves the structure of forest frames in the sense of next result. 

\begin{lemma} \label{key} Let $F = ({\bf X}, R^{\Box}, R^{\Diamond})$ be a forest frame. Then: 
\begin{itemize}
\item[(1)]
if $xR^{\Box}y$ then $k(x) R^{\Box}_{G(X)} k(y)$. 

\item[(2)] if $R^{\Box} \circ {\geq} = R^{\Box}$, then if $k(x) R^{\Box}_{G(X)} k(y)$ then $xR^{\Box}y$. 

\item[(3)] if $xR^{\Diamond}y$ then $k(x) R^{\Diamond}_{G(X)} k(y)$. 

\item[(4)] if $R^{\Diamond} \circ {\leq} = R^{\Diamond}$, then if $k(x) R^{\Diamond}_{G(X)} k(y)$ then $xR^{\Diamond}y$. 
\end{itemize}
\end{lemma}

\begin{proof} First of all, note that, by definition: \\

$k(x) R^{\Box}_{G(X)} k(y)$ iff $\beta^{-1}(k(x)) \subseteq k(y)$

\hspace{2.3cm} iff $\forall f \in G(X)$, $x \in \beta(f)$ implies $y \in f$

\hspace{2.3cm} iff $\forall f \in G(X)$, $(\forall z) (x R^{\Box} z$ implies $z \in f$) implies $y \in f$ \\

$k(x) R^{\Diamond}_{G(X)} k(y)$ iff $ k(y) \subseteq \delta^{-1}(k(x)) $

\hspace{2.3cm} iff $\forall f \in G(X)$, $y \in f$  implies $x \in \delta(f)$

\hspace{2.3cm} iff $\forall f \in G(X)$, $y \in f$ implies $(\exists z) (x R^{\Diamond} z$ and $z \in f)$  \\

\noindent (1) Assume $xR^{\Box}y$. Let $f \in G(X)$ such that  $(\forall z) (x R^{\Box} z$ implies $z \in f$). Hence, taking $z = y$, we get $y \in f$, and thus $k(x) R^{\Box}_{G(X)} k(y)$. 

\noindent (2) Assume $k(x) R^{\Box}_{G(X)} k(y)$. Let $f = R^{\Box}(x)$. Since $R^{\Box} \circ {\geq} = R^{\Box}$, $R^{\Box} (x)$ is a down-subset of $X$, i.e. $R^{\Box}(x) \in G(X)$. Therefore, $y \in R^{\Box}(x)$, i.e. $xR^{\Box}y$. 

\noindent (3) Assume $xR^{\Diamond}y$. Let $f \in G(X)$ such that $y \in f$. Then taking $z = y$, the condition $(\exists z) (x R^{\Diamond} z$ and $z \in f)$ is trivially satisfied.  Thus $k(x) R^{\Diamond}_{G(X)} k(y)$. 

\noindent (4) Assume $k(x) R^{\Diamond}_{G(X)} k(y)$. Let $f = {\downarrow} y = \{ z \in X \mid z \leq y\}$. Clearly, $f \in G(X)$ and $y \in f$. Then  $(\exists z) (x R^{\Diamond} z$ and $z \leq y)$. Since $R^{\Diamond} \circ {\leq} = R^{\Diamond}$, it follows that $x R^{\Diamond} y$. 
\end{proof}

\begin{corollary} \label{frameiso} Let $F = ({\bf X}, R^{\Box}, R^{\Diamond})$ be a forest frame such that $R^{\Box} \circ {\geq} = R^{\Box}$ and $R^{\Diamond} \circ {\leq} = R^{\Diamond}$. Then the frame $F = ({\bf X}, R^{\Box}, R^{\Diamond})$ is isomorphic to the frame $({\bf F_{G(X)}}, R^{\Box}_{G(X)}, R^{\Diamond}_{G(X)})$. 
\end{corollary}

In \cite[Theorem 7 (1-2)]{Cel08}, Celani showed what the effect of (D1) and (D2) is, for distributive modal algebras, on the relations $R_\Box$ and $R_\Diamond$ of their associated OR-frames. In the next result, we prove something similar. Indeed we prove what the effect of (D1), (D2) and (FS2) is on the forest-based representation of G\"odel algebras with operators by specific properties of the corresponding relations in forest frames.

 \begin{theorem}\label{keyD1}
 Let $({\bf X}, R_\Box, R_\Diamond)$ be a forest frame. Let $R'_\Box = R_\Box \circ {\geq}$, $R'_\Diamond = R_\Diamond \circ {\leq}$ and $R' = R'_\Box \cap R'_\Diamond$. Then
 \begin{enumerate}
 \item[(i)] $({\bf G(X)}, \beta, \delta)$ satisfies (D1) iff $R'_\Box = {R' \circ {\geq}}$
\item[(ii)] $({\bf G(X)}, \beta, \delta)$ satisfies (D2) iff $R'_\Diamond = {R' \circ {\leq}}$
\item[(iii)] $({\bf G(X)}, \beta, \delta)$ satisfies (FS2) iff $R'_\Box= {\geq\circ R'}$
 \end{enumerate}
 \end{theorem}
 
 \begin{proof}  First of all, due to Lemma \ref{diamond}, notice that the GAO generated by a forest frame $({\bf X}, R_\Box, R_\Diamond)$, $({\bf G(X)}, \beta, \delta)$, is the same a the one generated by the frame $({\bf X}, R'_\Box, R'_\Diamond)$, and hence we can apply Corollary \ref{frameiso} and get that $({\bf X}, R'_\Box, R'_\Diamond)$ is isomorphic to $({\bf F_{G(X)}}, R^{\Box}_{G(X)}, R^{\Diamond}_{G(X)})$. Then, letting $R_{G(X)}=R^\Box_{G(X)}\cap R^\Diamond_{G(X)}$, we have: 
 
 (i) By Theorem \ref{CJteo}, $({\bf G(X)}, \beta, \delta)$ satisfies (D1) iff $R^{\Box}_{G(X)} = R_{G(X)} \circ {\geq}$, and this holds, by Corollary \ref{frameiso}, iff $R'_\Box = R' \circ {\geq}$.  

 (ii) By Theorem \ref{CJteo}, $({\bf G(X)}, \beta, \delta)$ satisfies (D2) iff $R^{\Diamond}_{G(X)} = R_{G(X)} \circ {\leq}$, and this holds, by Corollary \ref{frameiso}, iff $R'_\Diamond = R' \circ {\leq}$. 
 
(iii) By Theorem \ref{CJteo}, $({\bf G(X)}, \beta, \delta)$ satisfies (FS2) iff $R^\Box_{G(X)} = {\geq \circ} R_{G(X)}$ and this holds, by Corollary \ref{frameiso}, iff $R'_\Box = {\geq\circ} R' $.
 \end{proof}

As we recalled at the beginning if this section, Dunn axioms have been mainly consideres in the frame of the positive fragment of classical (intuitionistic) logic. Axiomatic extensions of full intuitionistic modal logic by (D1) and (D2) have also been considered in \cite{Palmigiano2}, although e.g.\ in \cite{RodVid} the authors study the extension of bi-modal G\"odel logic by the axiom (D1). On the other hand, Fischer Servi axioms, for their own formulation that in fact needs the implication connective, have been quite deeply studied in the realm of intuitionistic modal logic, see \cite{Wolter0,Wolter}.
  
   We end this section with two examples showing that, indeed, $\mathbb{DGAO}\not\subseteq\mathbb{FSGAO}$ and $\mathbb{FSGAO}\not\subseteq\mathbb{DGAO}$, and hence showing that, in particular, (D1) and (FS2) are independent when considered in G\"odel modal logic and hence also on the modal intuitionistic basis a fortiori.
 
 The fist example is a GAO based on the free 1-generated G\"odel algebra that is a DGAO but it fails to prove (FS2).
 
  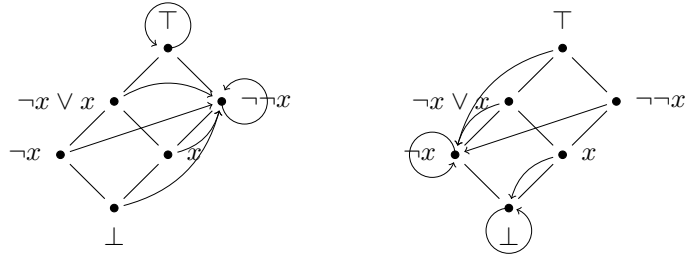
\begin{figure}[h]
 \begin{center}
 \begin{tikzpicture}

  \node [label=below:{${\bot}$}, label=below:{} ] (n1)  {} ;
  \node [above right of=n1,label=right:{${x}$}, label=below:{ }] (n2)  {} ;
  \node [above left of=n1,label=left:{${\neg x}$},label=below:{}] (n3) {} ;
  \node [above right of=n2,label=right:{${\neg \neg x}$},label=below:{}] (n4) {} ;
  \node [above left of=n4,label=above:{${\top}$},label=below:{}] (n6) {} ;
  \node [above right of=n3,label=left:{${\neg x \vee x}$},label=below:{}] (n5) {} ;

  \draw  (n1) -- (n2);
  \draw (n1) -- (n3);
  \draw  (n3) -- (n5);
  \draw  (n2) -- (n4);
  \draw  (n4) -- (n6);
  \draw  (n5) -- (n6);
  \draw  (n2) -- (n5);
  \draw [fill] (n1) circle [radius=.5mm];
  \draw [fill] (n2) circle [radius=.5mm];
  \draw [fill] (n3) circle [radius=.5mm];
  \draw [fill] (n4) circle [radius=.5mm];
  \draw [fill] (n5) circle [radius=.5mm];
  \draw [fill] (n6) circle [radius=.5mm];

  \draw [line width=0.07mm,->] (n2) edge[bend right] (n4);
    \draw [line width=0.07mm, ->] (n3) --(n4);
      \draw [line width=0.07mm, ->] (n5) edge[bend left] (n4);
        \draw [line width=0.07mm, ->] (n1) edge[bend right] (n4);
         \draw [line width=0.07mm, ->]  (n6) arc [radius=3mm, start angle=270, end angle=600 ]  (n6);
         \draw [line width=0.07mm, ->]  (n4) arc [radius=3mm, start angle=180, end angle=510 ]  (n4);
\end{tikzpicture}
\hspace{1cm}
\begin{tikzpicture}

  \node [label=below:{${\bot}$}, label=below:{} ] (n1)  {} ;
  \node [above right of=n1,label=right:{${x}$}, label=below:{ }] (n2)  {} ;
  \node [above left of=n1,label=left:{${\neg x}$},label=below:{}] (n3) {} ;
  \node [above right of=n2,label=right:{${\neg \neg x}$},label=below:{}] (n4) {} ;
  \node [above left of=n4,label=above:{${\top}$},label=below:{}] (n6) {} ;
  \node [above right of=n3,label=left:{${\neg x \vee x}$},label=below:{}] (n5) {} ;

  \draw  (n1) -- (n2);
  \draw (n1) -- (n3);
  \draw  (n3) -- (n5);
  \draw  (n2) -- (n4);
  \draw  (n4) -- (n6);
  \draw  (n5) -- (n6);
  \draw  (n2) -- (n5);
  \draw [fill] (n1) circle [radius=.5mm];
  \draw [fill] (n2) circle [radius=.5mm];
  \draw [fill] (n3) circle [radius=.5mm];
  \draw [fill] (n4) circle [radius=.5mm];
  \draw [fill] (n5) circle [radius=.5mm];
  \draw [fill] (n6) circle [radius=.5mm];

  \draw [line width=0.07mm, ->] (n2) edge[bend right] (n1);
    \draw [line width=0.07mm, ->] (n3) arc [radius=3mm, start angle=0, end angle=340 ]  (n3);
        \draw [line width=0.07mm, ->] (n5) edge[bend right] (n3);
            \draw [ line width=0.07mm,->] (n4) --(n3);
                \draw [line width=0.07mm, ->] (n6) edge[bend right] (n3);
   \draw [line width=0.07mm, ->]  (n1) arc [radius=3mm, start angle=90, end angle=430 ]  (n1);
\end{tikzpicture}\caption{A finite G\"odel algebra with a $\Box$ (left-hand-side) and a $\Diamond$ (right-hand-side) that satisfies Dunn's axioms but it does not satisfy (FS2).}
\label{figDnotFS}
\end{center}
\end{figure}

Define, on $\free(1)$, the following operators as in Figure \ref{figDnotFS}: 
\begin{center}
$\Box\top=\top$; $\Box y=\neg\neg x$ for all $y\neq \top$; $\Diamond \bot= \Diamond x=\bot$; $\Diamond\neg x=\Diamond(x\vee \neg x)=\Diamond\neg\neg x=\Diamond\top=\neg x$. 
\end{center}
Then, one has, $\Diamond x\to \Box\bot=\bot\to\neg\neg x=\top\not\leq \Box(x\to\bot)=\neg\neg x$ and hence (FS2) fails. 

However, (D1) and (D2) holds. In fact, as for (D1), notice that, for all $a, b$ such that $a\vee b\neq \top$, $\Box(a\vee b)=\neg\neg x$ and in these cases  $\Box a=\Box b=\neg\neg x$. Thus, $\Box(a\vee b)= \Box a\leq \Box a \vee \Diamond b$. Now, if $a\vee b=\top$ and avoiding the trivial case in which either $a$ or $b$ equals $\top$, one has that either $a$ or $b$ are $\neg\neg x$. Assume $a=\neg\neg x$. Then, if $b=\neg x$, $\Box (\neg\neg x\vee \neg x)=\top=\Box \neg\neg x\vee \Diamond \neg x=\neg\neg x\vee \neg x$. Conversely, if $b=\neg x\vee x$, again $\Box(\neg\neg x\vee (\neg x\vee x))=\top=\Box \neg\neg x\vee \Diamond (\neg x\vee x)=\neg\neg x\vee \neg x$. 

That (D2) also holds can be proved in a similar manner and the proof is omitted. 

As for the second example, consider the (directly indecomposable) G\"odel algebra ${\bf A}$ of Figure \ref{figFSnotD} below where $\Box$ and $\Diamond$ are so defined:
\begin{center}
$\Box \top=\Box d=\top$; $\Box c=\Box b=a$; $\Box a=\Box\bot=\bot$; $\Diamond\top=\Diamond d=\Diamond c= \Diamond b=\Diamond a=c$; $\Diamond\bot=\bot$. 
\end{center}

 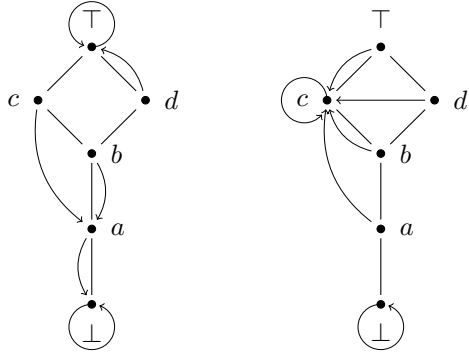
\begin{figure}[h]
\begin{center}
\begin{tikzpicture}

\node [label=below:{${\bot}$}, label=below:{} ] (n1)  {} ;
\node [above of=n1,label=right:{${a}$}, label=below:{ }] (n2)  {} ;
\node [above  of=n2,label=right:{${b}$}, label=below:{ }] (n3)  {} ;
\node [above right of=n3,label=right:{${d}$}, label=below:{ }] (n4)  {} ;
\node [above left of=n3,label=left:{${c}$}, label=below:{ }] (n5)  {} ;
\node [above right of=n5,label=above:{${\top}$}, label=below:{ }] (n6)  {} ;

  \draw  (n1) -- (n2);
  \draw  (n2) -- (n3);
    \draw  (n3) -- (n4);
      \draw  (n3) -- (n5);
        \draw  (n5) -- (n6);
          \draw  (n4) -- (n6);  
    \draw [fill] (n1) circle [radius=.5mm];
  \draw [fill] (n2) circle [radius=.5mm];
  \draw [fill] (n3) circle [radius=.5mm];
  \draw [fill] (n4) circle [radius=.5mm];
  \draw [fill] (n5) circle [radius=.5mm];
  \draw [fill] (n6) circle [radius=.5mm];
    \draw [line width=0.07mm, ->] (n4) edge[bend right] (n6);
    \draw [line width=0.07mm, ->] (n6) arc [radius=3mm, start angle=270, end angle=610 ]  (n6);
        \draw [line width=0.07mm, ->] (n3) edge[bend left] (n2);
    \draw [line width=0.07mm, ->] (n5) edge[bend right] (n2);
        \draw [line width=0.07mm, ->] (n2) edge[bend right] (n1);
    \draw [line width=0.07mm, ->] (n1) arc [radius=3mm, start angle=90, end angle=430 ]  (n1);

 
   \end{tikzpicture}
\hspace{1cm}
\begin{tikzpicture}
\node [label=below:{${\bot}$}, label=below:{} ] (n1)  {} ;
\node [above of=n1,label=right:{${a}$}, label=below:{ }] (n2)  {} ;
\node [above  of=n2,label=right:{${b}$}, label=below:{ }] (n3)  {} ;
\node [above right of=n3,label=right:{${d}$}, label=below:{ }] (n4)  {} ;
\node [above left of=n3,label=left:{${c}$}, label=below:{ }] (n5)  {} ;
\node [above right of=n5,label=above:{${\top}$}, label=below:{ }] (n6)  {} ;

  \draw  (n1) -- (n2);
  \draw  (n2) -- (n3);
    \draw  (n3) -- (n4);
      \draw  (n3) -- (n5);
        \draw  (n5) -- (n6);
          \draw  (n4) -- (n6);  
    \draw [fill] (n1) circle [radius=.5mm];
  \draw [fill] (n2) circle [radius=.5mm];
  \draw [fill] (n3) circle [radius=.5mm];
  \draw [fill] (n4) circle [radius=.5mm];
  \draw [fill] (n5) circle [radius=.5mm];
  \draw [fill] (n6) circle [radius=.5mm];
    \draw [line width=0.07mm, ->] (n2) edge[bend left] (n5);
    \draw [line width=0.07mm, ->] (n5) arc [radius=3mm, start angle=0, end angle=320 ]  (n5);
        \draw [line width=0.07mm, ->] (n3) edge[bend left] (n5);
    \draw [line width=0.07mm, ->] (n4) -- (n5);
        \draw [line width=0.07mm, ->] (n6) edge[bend right] (n5);
    \draw [line width=0.07mm, ->] (n1) arc [radius=3mm, start angle=90, end angle=430 ]  (n1);
   \end{tikzpicture}
\caption{A finite G\"odel algebra with a $\Box$ (left-hand-side) and a $\Diamond$ (right-hand-side) that satisfies Fischer Servi's axioms but it does not satisfy (D1).}
\label{figFSnotD}
\end{center}
\end{figure}

 In that GAO, one has: $\Box (a\vee d)=\Box d=\top\not\leq \Box a\vee \Diamond d=\bot\vee c=c$ and hence (D1) fails. On the other hand, it satisfies (FS1) and (FS2). In order to prove that, consider the forest (tree) $\bF_{\bf A}$ of the prime filters of ${\bf A}$  as in Figure \ref{figTree1} 
 \begin{figure}[h]
 \begin{center}
 \begin{tikzpicture}
\node [label=right:{${x}$}, label=below:{ }] (n1)  {} ;
\node [above  of=n1,label=right:{${y}$}, label=below:{ }] (n2)  {} ;
\node [above right of=n2,label=right:{${k}$}, label=below:{ }] (n3)  {} ;
\node [above left of=n2,label=left:{${z}$}, label=below:{ }] (n4)  {} ;

  \draw  (n1) -- (n2);
  \draw  (n2) -- (n3);
    \draw  (n2) -- (n4);
     
    \draw [fill] (n1) circle [radius=.5mm];
  \draw [fill] (n2) circle [radius=.5mm];
  \draw [fill] (n3) circle [radius=.5mm];
  \draw [fill] (n4) circle [radius=.5mm];
   \end{tikzpicture}
      \caption{The tree of prime filters of the G\"odel algebra ${\bf A}$ of Figure \ref{figFSnotD}.}
    \label{figTree1}
 \end{center}
 \end{figure}
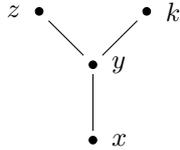
where $x$ denotes the principal filter generated by $a$; $y$ the principal filter generated by $b$; $z$ is the principal filter generated by $c$ and $k$ is the principal filter generated by $d$. Then,   
 compute the relations $R_\Box$ and $R_\Diamond$:
$$
R_\Box=\{(x,x),(x,y),(y,x),(y,y),(y,z),(z,x),(z,y),(z,z),(k,x),(k,y),(k,z)\}
$$
and 
$$
R_\Diamond=\{(x,x),(x,y),(x,z),(x,k),(y,x),(y,y),(y,z),(y,k),(k,x),(k,y),(k,z),(k,k)\}
$$
so that 
$$
R=R_\Box\cap R_\Diamond=\{(x,x),(x,y),(y,x),(y,y),(y,z),(k,x),(k,y),(k,z)\}
$$
It is not difficult to prove that $(\bF_{\bf A}, R)$ is an {\em IK-frame} in the sense of \cite{Ale} and therefore the algebra $({\bf G}(\bF_{\bf A}), \beta_{R}, \delta_R)$ is isomorphic to $({\bf A}, \Box, \Diamond)$ and it satisfies (FS1) and (FS2).


\section{Forest frames with a single relation} \label{one-relation}

In this section we review results in the literature about relational frames for Heyting and positive modal algebras with a single binary relation, and adapt them to our setting to finite GAOs. 
In particular, we will consider results by Palmigiano \cite{Ale} and Orlowska and Rewitzky \cite{Orlo} on dualities for Intuitionistic modal logics, as well as Celani and Jansana's results on duality for positive modal logic \cite{CelJan}. In the last part of the section we consider special forest frames in which the binary relation satisfies both the properties of monotonicity  and antimonotonicity in the first argument.

%
%
%
%
%
%

For the remining of this section, it is useful to recall the above Lemma \ref{bijective} showing that every finite forest is isomorphic to the forest of prime filters of its associated G\"odel algebra.
%
%
%
Furthermore, recall from Section \ref{sec:tworelations} that, for each GAO, $({\bf A}, \Box, \Diamond)$, we can consider the associated relational frame $({\bf F_A}, R_A)$, where  $R_A$ is the binary relation on $F_A$ defined as follows: for every $f, g \in F_A$, 
$$ f  R_A g   \quad \mbox { iff } \quad \Box^{-1}(f) \subseteq g \subseteq \Diamond^{-1}(f) .$$
Equivalently, one can define $R_A$ as the intersection of the two relations $R^\Box_A$ and $R^\Diamond_A$, i.e.\ $R_A = R^\Box_A \cap R^\Diamond_A$, where $f  R^\Box_A g$ if  $\Box^{-1}(f) \subseteq g$, and  $ f  R^\Diamond_A g$ if  $\Diamond(g) \subseteq f$.

\subsection{Forest frames for GAOs satisfying Fisher-Servi axioms} 

Palmigiano in \cite{Ale} and Or{\l}owska and Rewitzky in \cite{Orlo}  study duality theory for some intuitionistic modal logics. It is worth noticing that the notions of relational frames with a single relation in these papers that are relevant for our setting, namely the so-called IK-frames in \cite{Ale} and HK1-frames in \cite{Orlo}, coincide. They are shown to capture Heyting modal algebras satisfying the Fischer Servi axioms. 

As already recalled, G\"odel logic is the axiomatic extension of Intuitionistic logic with the pre-linearity axiom, and hence the variety of G\"odel algebras are the subvariety of Heyting algebras generated by the linearly-ordered ones. Thus, again, many results in \cite{Ale} and \cite{Orlo} also extend to our finitary setting of G\"odel algebras with operators, the basic difference being at the level of relational frames is that we consider here frames on forests rather than on preordered sets. 

\begin{definition} A relational frame $F = ({\bf X}, R)$ is a {\em FS-forest frame} provided that $\bf X$ is a forest and the binary relation $R \subseteq X \times X$ satisfies the following two conditions:  

(FS1) $(\cleq \circ R) \subseteq (R \circ \cleq)$

(FS2) $(R \circ \cgeq) \subseteq (\cgeq \circ R)$.

\end{definition}

Note that condition (FS1) is in fact the same as (CJ1), and conditions (FS1) and (FS2), that appear in the definition of IK-frames in \cite{Ale}, are respectively equivalent to the conditions that appear in the definition of HK1-frames in \cite{Orlo}.

\begin{lemma} \label{newxx} The conditions (FS1) and (FS2) are respectively equivalent  to: 

{\em (FS1')} $(\cleq \circ R \circ \cleq) = (R \circ \cleq)$  

{\em (FS2')} $(\cgeq \circ R \circ \cgeq) = (\cgeq \circ R)$. 
\end{lemma}

\begin{proof} First of all, note that (FS1') and (FS2') are respectively equivalent to 

{\em (FS1'')} $(\cleq \circ R \circ \cleq) \subseteq (R \circ \cleq)$  

{\em (FS2'')} $(\cgeq \circ R \circ \cgeq) \subseteq ( \cgeq \circ R )$.

\noindent since the reverse inclusions always hold. Now let us prove that (FS2) is equivalent to (FS2''): 
\begin{itemize}
\item[-] Assume $(R \circ \cgeq) \subseteq (\cgeq \circ R)$ holds.  Then, by monotonicity (wrt set inclusion) of the composition,  $(\cgeq \circ R \circ \cgeq) \subseteq (\cgeq \circ \cgeq \circ R)$, but $\cgeq \circ \cgeq = \cgeq$, and thus, $(\cgeq \circ R \circ \cgeq) \subseteq (\cgeq \circ R)$. 
\item[-] Conversely, assume $(\cgeq \circ R \circ \cgeq) \subseteq  (\cgeq \circ R)$. But, trivially, $(R \circ \cgeq) \subseteq (\cgeq \circ R \circ \cgeq)$, and hence, by transitivity, $(R \circ \cgeq) \subseteq (\cgeq \circ R)$. 
\end{itemize}
The case of (FS1) and (FS1'') can be proved analogously. 
\end{proof}

%
%
%
%
%
%

\noindent Then, if $F = ({\bf X}, R)$ is a FS-forest frame, then the operations $\beta'_R = \beta_{\cgeq \circ R}$ and $\delta_R$ on subsets of $X$ are closed on the set $G(X)$ of downsets of $(X, \leq)$ and adapting the results in \cite{Ale,Orlo}, we have that  the modal algebra
$${\bf G}'(F) = ({\bf G(X)}, \beta'_{R}, \delta_R)$$
is a FS-GAO. Now, ${\bf G}'(F)$ induces its associated  relational frame $({\bf F_{G(X)}}, R_{G'(F)})$ as defined above. Next proposition shows that we recover the initial frame $F$ up to a isomorphism. 

\begin{proposition}  \label{FSP1} 
For every FS-forest frame $F=({\bf X}, R)$, $({\bf F_{G(X)}}, R_{G'(F)})$ is a FS-forest frame and, moreover, $({\bf X}, R) \cong ({\bf F_{G(X)}}, R_{G'(F)})$.
\end{proposition}

\begin{proof} That $({\bf F_{G(X)}}, R_{G'(F)})$ is a FS-forest frame directly follows from (3) of \cite[Corollary 5.7]{Orlo}. Let the mapping $k : X \to F_{G(X)}$ be defined, for any $x \in X$, by $k(x) = \{f \in G(X) \mid x \in f\}$. Then  Lemma \ref{bijective}
shows that $k$ is bijective and order preserving, while \cite[Lemma 4.5]{Orlo} shows that, for all $x, y \in X$, $xRy$ iff $k(x) R_{G'(F)} k(y)$. 
\end{proof}

\begin{proposition} \label{FSP2} For any FS-frame $F = ({\bf X}, R)$, let $R' = (\cgeq \circ R ) \cap (R \circ \cleq)$. Then: 
\begin{itemize}
\item[(i)]  $R'$ satisfies (FS1) and (FS2)
\item[(ii)]  $(\cgeq \circ R') = (\cgeq \circ R)$,  $(R' \circ \cleq) = (R \circ \cleq)$
\item[(iii)]  $R' = (\cgeq \circ R' ) \cap (R' \circ \cleq)$
\item[(iv)]  $\beta_{\geq \circ R'} = \beta_{\geq \circ R}$, $\delta_{R'} = \delta_{R}$
\end{itemize}
\end{proposition}

\begin{proof} \begin{itemize}

\item[(i)] As for (FS1), the proof is practically the same than for (CJ1) in (i) of Prop. \ref{CJP2}. And as for (FS2) we have $(R' \circ \cgeq) = (((\cgeq \circ R ) \cap (R \circ \cleq)) \circ \cgeq) \subseteq (\cgeq \circ R \circ \cgeq) = (\cgeq \circ R)  \subseteq  (\cgeq \circ R' )$. 

\item[(ii)] The inclusions $\supseteq$'s are direct, let us prove the inclusions $\subseteq$'s. We have: $(R' \circ \cgeq) =  (((\cgeq \circ R) \cap (R \circ \cleq)) \circ \cgeq) \subseteq  ((\cgeq \circ R) \circ \cgeq) = (\cgeq \circ R)$, and similarly $(R' \circ \cleq) =  (((\cgeq \circ R) \cap (R \circ \cleq)) \circ \cleq) \subseteq  ((R \circ \cleq) \circ \cleq) = (R \circ \cleq)$. 

\item[(iii)] It directly follows from the definition of $R'$ and (ii). 

\item[(iv)] It directly follows from (ii) and the fact that  $\delta_R = \delta_{R\circ \cleq}. $
\end{itemize}
\end{proof}

We now present a final result on FS-forest frames that 
a direct consequence of the  properties proved in Proposition \ref{FSP1} and Proposition \ref{FSP2} above. 

\begin{corollary}  For every FS-forest frame $F = ({\bf X}, R)$, let $R' = (\cgeq \circ R ) \cap (R \circ \cleq)$. Then $F' = ({\bf X}, R')$ is a FS-forest frame that is equivalent to $F = ({\bf X}, R)$, i.e. ${\bf G}'(F') ={\bf G}(F)$. 
\end{corollary}

\subsection{Forest frames for GAOs satisfying Dunn axioms}

In \cite{CelJan} Celani and Jansana study the duality theory for Dunn's positive modal logic \cite{Dunn}. Positive modal algebras can always be expanded with a implication operation such that the resulting structure is a Dunn-GAO. Many results can be easily extended to our finitary setting of G\"odel algebras with operators. 

\begin{definition} A relational frame $F = ({\bf X}, R)$ is a {\em CJ-forest frame} provided that $\bf X$ is a forest and the binary relation $R \subseteq X \times X$ satisfies the following two conditions:  \vspace{0.1cm}

(CJ1): $(\cleq \circ R) \subseteq (R \circ \cleq)$

(CJ2): $(\cgeq \circ R) \subseteq (R \circ \cgeq)$.
\end{definition}
%

The proof of the following lemma is very similar to that of Lemma \ref{newxx} and it is omitted.

\begin{lemma}  \label{CJcond} Conditions (CJ1) and (CJ2) are respectively equivalent  to: 

{\em (CJ1')} $(\cleq \circ R \circ \cleq) = (R \circ \cleq)$  

{\em (CJ2')} $(\cgeq \circ R \circ \cgeq) = (R \circ \cgeq)$. 

\end{lemma}

%
%
%
%
%
%

\noindent Moreover, if $F = ({\bf X}, R)$ is a CJ-forest frame, then the operations $\beta_R$ and $\delta_R$ on subsets of $X$ are in fact closed on the set $G(X)$ of downsets of $(X, \leq)$.  Then, based on \cite{CelJan}, one can check that the modal algebra 
$${\bf G}(F) = ({\bf G(X)}, \beta_R, \delta_R)$$
is a Dunn-GAO.  Hence, by Proposition \ref{bool}, the operators  $\beta_R$ and $\delta_R$ are closed on the set of Boolean elements of ${\bf G(X)}$.

It is also very interesting to observe that any CJ-forest frame is equivalent to a basic forest frame with two (different) relations in the sense of generating the same algebra. Indeed, given a CJ-frame  $F = ({\bf X}, R)$, let us consider the following two relations: $R_\Box = R \circ \cgeq$ and $R_\Diamond =  R \circ \cleq$. Then in \cite{CelJan} the authors prove that
$$\beta_R = \beta_{R_\Box}, \quad \delta_R = \delta_{R_\Diamond}. $$
Now, consider the intersection of these two relations  $R' = R_\Box \cap R_\Diamond = (R \circ \cgeq) \cap (R \circ \cleq)$. Clearly, $R \subseteq R'$, and it is not hard to prove the following further properties. 

\begin{proposition} \label{CJP2} For any CJ-forest frame $F = ({\bf X}, R)$, define a new relation $R' = (R \circ \cgeq) \cap (R \circ \cleq)$. Then: 
\begin{itemize}
\item[(i)] $R'$ satisfies (CJ1) and (CJ2)
\item[(ii)] $(R' \circ \cgeq) = (R \circ \cgeq)$,  $(R' \circ \cleq) = (R \circ \cleq)$, i.e. $R'_\Box = R_\Box$ and $R'_\Diamond = R_\Diamond$ 
\item[(iii)] $R'$ satisfies (FS2)
\item[(iv)] $\cgeq \circ  (R \circ \cgeq) = R \circ \cgeq$,  $\cleq \circ  (R \circ \cleq) = R \circ \cleq$, i.e. $\cgeq \circ R_\Box = R_\Box$ and $\cleq \circ R_\Diamond = R_\Diamond$
\item[(v)] $R' = (R' \circ \cgeq) \cap (R' \circ \cleq)$, i.e. $R' = R'_\Box \cap R'_\Diamond$
\item[(vi)] $\beta_{R'} = \beta_{R}$, $\delta_{R'} = \delta_{R}$
\end{itemize}
\end{proposition}

\begin{proof} \begin{itemize}

\item[(i)] As for (CJ1), we have  $(\cleq \circ R') = (\cleq \circ ( (R \circ \cgeq) \cap (R \circ \cleq)))  \subseteq (\cleq \circ R \circ \cleq) = (R \circ \cleq) \subseteq  (R' \circ \cleq)$. And as for (CJ2) we have $(\cgeq \circ R') = (\cgeq \circ ( (R \circ \cgeq) \cap (R \circ \cleq))) \subseteq  (\cgeq \circ R \circ \cgeq) = (R \circ \cgeq) \subseteq  (R' \circ \cgeq)$. 

\item[(ii)] The inclusions $\supseteq$'s are direct, let us prove the inclusions $\subseteq$'s. We have: $(R' \circ \cgeq) =  (((R \circ \cgeq) \cap (R \circ \cleq)) \circ \cgeq) \subseteq  ((R \circ \cgeq) \circ \cgeq) = (R \circ \cgeq)$, and similarly $(R' \circ \cleq) =  (((R \circ \cgeq) \cap (R \circ \cleq)) \circ \cleq) \subseteq  ((R \circ \cleq) \circ \cleq) = (R \circ \cleq)$. 

\item[(iii)] By (ii), $(R' \circ \cgeq) = (R \circ \cgeq)$, but by (and since $R  \subseteq R'$ then we have $(R \circ \cgeq) \subseteq (R'  \circ \cgeq)$, and hence 
$(R' \circ \cgeq)  \subseteq (R'  \circ \cgeq)$. 

\item[(iv)] These are exactly properties (CJ2') and (CJ1'), respectively.

\item[(v)] It directly follows from the definition of $R'$ and (ii). 

\item[(vi)] It directly follows from (ii) and the fact that $\beta_R = \beta_{R_\Box}$ and $\delta_R = \delta_{R_\Diamond}. $
\end{itemize}
\end{proof}

Observe that properties (i) and (iii) above tells us respectively that the frame $F' = ({\bf X}, R')$ is both a CJ-frame, and a basic frame in the sense of Definition \ref{basic}. Moreover, by property (vi), it generates the same Dunn-GAO than the original frame $F = ({\bf X}, R)$.


\begin{corollary} For every CJ-forest frame $F=({\bf X}, R)$, $F' = ({\bf X}, R')$ is a CJ-forest frame that is equivalent to $F = ({\bf X}, R)$, i.e. ${\bf G}(F) ={\bf G}(F')$. 
\end{corollary}

Finally, similarly to the case of FS-frames, starting from the algebra ${\bf G}(F)$, one can consider its associated relational frame $({\bf F_{G(X)}}, R_{G(F)})$. However, next proposition shows that, unlike the case of FS-frames, in general we do not recover the initial frame $F$, but the modified frame $F'$.

\begin{proposition}  \label{CJP1}
For every CJ-forest frame $F = ({\bf X}, R)$ and its associated frame $F' = ({\bf X}, R')$, 
$({\bf F_{G(X)}}, R_{G(F)})$ is a CJ-forest frame and, moreover, $F' = ({\bf X}, R') \cong ({\bf F_{G(X)}}, R_{G(F)})$.
\end{proposition}

\begin{proof} That $({\bf F_{G(X)}}, R_{G(F)})$ is a CJ-forest frame directly follows from \cite[Lemma 2.1]{CelJan}. By Lemma \ref{bijective}, the mapping $k : X \to F_{G(X)}$, defined as $k(x) = \{f \in G(X) \mid x \in f\}$ for any $x \in X$, is a bijection. Thus, we are left to prove that $x, y \in X$, $xR'y$ iff $k(x) R_{G(F)} k(y)$. But, by defining as above $R_\Box = R \circ \cgeq$ and $R_\Diamond \circ \cleq$, it turns out that $R_\Box = R_\Box \circ \cgeq$ and $R_\Diamond = R_\Diamond \circ \cleq$, and hence we can apply Lemma \ref{key} and get, for all $x, y \in X$: \vspace{0.1cm}

(i) $x R_\Box y$ iff $k(x) R_\Box^{G(F)} k(y)$ \vspace{0.1cm}

(ii) $x R_\Diamond y$ iff $k(x) R_\Diamond^{G(F)} k(y)$ \vspace{0.1cm}

\noindent Now, by definition of $R'$, $xR'y$ iff $x R_\Box y$ and $x R_\Diamond y$, and by (i) and (ii), this holds iff $k(x) R_\Box^{G(F)} k(y)$ and $k(x) R_\Diamond^{G(F)} k(y)$, but this is just $k(x) R_{G(F)} k(y)$.      
\end{proof}

\subsection{Forest frames for GAOs satisfying both Dunn and Fischer Servi axioms} \label{DGAO}

In this section we consider the class of frames that are both CJ- and FS-forest frames. 

\begin{definition} A relational frame $F = ({\bf X}, R)$ is a {\em FSD-forest frame} provided that $\bf X$ is a forest and the binary relation $R \subseteq X \times X$ satisfies the following two conditions:  

(FS1) $(\cleq \circ R) \subseteq (R \circ \cleq)$

(FS2) $(R \circ \cgeq) \subseteq (\cgeq \circ R)$

(CJ2) $(\cgeq \circ R) \subseteq (R \circ \cgeq)$.

\end{definition}

By definition, it is clear that $F = ({\bf X}, R)$ is a FSD-forest frame iff $F$ is both a CJ-forest frame and a FS-forest frame. 

It is easy to check that requiring the above  three conditions is equivalent to require the following two conditions:
\begin{itemize}
\item[] (FS1') $(R \circ \cleq) = (\cleq \circ R \circ \cleq)$

(FSCJ2) $(R \circ \cgeq) = (\cgeq \circ R)$
\end{itemize}
It is clear that (FS1') is a simple reformulation of (FS1), which is commonly satisfied by both CJ- and FS-forest frames, while (FSCJ2) is obtained by combining (FS2) and (CJ2).

In this case, notice that if $F = ({\bf X}, R)$ is a FSD-forest frame, then $\beta_R = \beta_{\cgeq \circ R}$, and thus the G\"odel modal algebra
$${\bf G}(F) = ({\bf G(X)}, \beta_R, \delta_R)$$
is both a FS-GAO and a Dunn-GAO. 
Note that, as in the case of CJ-forest frames, the operators $\beta_R$ and $\delta_R$ in ${\bf G}(F)$ keep being closed on the set of Boolean elements of ${\bf G(X)}$. 

Similarly to the previous cases, now we have the following two propositions that can be easily proved by  combining respectively Props \ref{CJP1} and \ref{CJP2} on the one hand and Props. \ref{FSP1} and \ref{FSP2} on the other. 

\begin{proposition} Let  $F=({\bf X}, R)$ be a FSD-forest frame and let $({\bf F_{G(X)}}, R_{G(F)})$ be the  frame such that ${\bf F_{G(X)}}$ is the forest of prime filters of  ${\bf G}(F)$ and $R_{G(F)}=R_{G(F)}^\Box\cap R_{G(F)}^\Diamond$. Then 
$({\bf F_{G(X)}}, R_{G(F)})$ is a FSD-forest frame and, moreover, $F = ({\bf X}, R) \cong ({\bf F_{G(X)}}, R_{G(F)})$.
\end{proposition}


\begin{proposition} Let $F = ({\bf X}, R)$ be a FSD-forest frame, and define a new relation 
$R' = (\cgeq \circ R ) \cap (R \circ \cleq)  =  (R \circ \cgeq ) \cap (R \circ \cleq)$. Then: 
\begin{itemize}
\item[(i)]  $R'$ satisfies (FS1), (FS2) and (CJ2)
\item[(ii)]  $(\cgeq \circ R') = (\cgeq \circ R)$,  $(R' \circ \cleq) = (R \circ \cleq)$
\item[(iii)]   $R' = (\cgeq \circ R' ) \cap (R' \circ \cleq) =  (R' \circ \cgeq ) \cap (R' \circ \cleq) $
\item[(iv)]  $\beta_{\geq \circ R'} = \beta_{\geq \circ R} = \beta_R$, $\delta_{R'} = \delta_{R}$
\end{itemize}
\end{proposition}
\begin{proof}
The proof of this proposition follows straightforwardly from Props. \ref{CJP2} and \ref{FSP2} by noticing that the relations $R'$ defined in those propositions coincide in a FSD-forest frame, i.e. if $R$ satisfies (FSJ2)  then $(\cgeq \circ R ) = (R \circ \cgeq )$ and hence, $ (\cgeq \circ R ) \cap (R \circ \cleq)  =  (R \circ \cgeq ) \cap (R \circ \cleq)$ and all the properties in Props.  \ref{FSP2}  and \ref{CJP2} are valid in a a FSD-forest frame. 
\end{proof}

\begin{corollary}  For every FSD-forest frame $F = ({\bf X}, R)$, let $R' = (\cgeq \circ R ) \cap (R \circ \cleq)$. Then $F' = ({\bf X}, R')$ is a FSD-forest frame that is equivalent to $F = ({\bf X}, R)$, i.e. ${\bf G}'(F') ={\bf G}(F)$. 
\end{corollary}

\subsection{Forest frames with one relation satisfying (A) and (M)} \label{MDGAO} 

In this section we finally consider forest frames with a single relation satisfying both monotonicity properties (A) and (M) on the first variable. As we will show, this class of frames determine a proper subvariety of 
 $\mathbb{DGAO}$.

\begin{definition} A relational frame $F = ({\bf X}, R)$ is a W-forest frame provided that $\bf X$ is a forest and the binary relation $R \subseteq X \times X$ satisfies the following two conditions:  

(W1): $(\cleq \circ R) \subseteq R $

(W2): $(\cgeq \circ R) \subseteq R$

\end{definition}

\noindent It is easy to check that conditions (W1) and (W2) are equivalent to the compound condition

 \begin{itemize}

\item[(W):] $(\cleq \circ R) = R = (\cgeq \circ R)$
\end{itemize}
Moreover, in a W-forest frame we further have all the following relations:
$$ R = (\cleq \circ R) = (\cgeq \circ R) \subseteq \left \{ 
\begin{array}{l}
\subseteq (R \circ \cleq) = (\cleq \circ R \circ \cleq) = (\cgeq \circ R \circ \cleq)  \vspace{0.2cm} \\
\subseteq (R \circ \cgeq) = (\cleq \circ R \circ \cgeq) = (\cgeq \circ R \circ \cgeq) \\
\end{array} 
\right . $$
Also note the following.

\begin{remark}\label{remCJW}
\begin{itemize}

\item[(i)] W-forest frames are CJ-forest frames, since (W1) and (W2) imply (CJ1) and (CJ2) respectively, but the converse is not true. 



\item[(ii)]  W-forest frames are not FS-forest frames in general, but if $F = ({\bf X}, R)$ is a W-forest frame, then $F' = ({\bf X}, R')$, where $R' = R \circ \cgeq$, is both a W-forest frame and a FS-forest frame. 

Indeed, we have: 

- (W1) $\cleq \circ R' = \cleq \circ R \circ \cgeq = R \circ \cgeq = R'$,

- (W2) $\cgeq \circ R' = \cgeq \circ R \circ \cgeq = R \circ \cgeq = R'$,

- (FS2) $R' \circ \cgeq = R  \circ \cgeq  \circ \cgeq = R  \circ \cgeq = R' \subseteq \cgeq \circ R'$.

\item[(iii)] From (i) and (ii) it follows that if $F = ({\bf X}, R)$ is a W-forest frame, since $\cgeq\circ R' = R'$, then $ \beta_{R'}$ and $\delta_{R'}$ are closed on $Down(X)$ and 
$${\bf G}(F') = ({\bf G(X)}, \beta_{R'}, \delta_{R'})$$
is both a Dunn-GAO and a FS-GAO. Notice that, in general ${\bf G}(F')$ is not isomorphic to the GAO ${\bf G}(F)$, for otherwise, every algebra defined by a W-forest frame would belong to $\mathbb{FSGAO}$ and this is not the case because of next Proposition \ref{propFinal} (iv). 

\end{itemize} 
\end{remark}
Next, we axiomatise the subvariety of Dunn-GAOs whose associated frames are W-forest frames. 

\begin{definition}\label{defWGAO} An algebra $({\bf A}, \Box, \Diamond)$ is a $W$-GAO if it is a D-GAO that satisfies the following two equations:
\begin{itemize}
\item[(BB)] $\Box(x) \lor \neg \Box(x) = 1$,
\item[(DB)] $\Diamond(x) \lor \neg \Diamond(x) = 1$.
\end{itemize}
\end{definition}
The class of $W$-GAOs is a variety, denoted $\mathbb{WGAO}$, that is a subvariety of  $\mathbb{DGAO}$. 

\begin{theorem} (1) Let $F = ({\bf X}, R)$ be a W-forest frame. Then ${\bf G}(F) = ({\bf G(X)}, \beta_R, \delta_R)$ is a $W$-GAO. 

(2) Let  $({\bf A}, \Box, \Diamond)$ be a $W$-GAO. Then $(F_{\bf A}, R_{\bf A})$ is a W-forest frame. 
\end{theorem}

\begin{proof}
(1) Let $F = (X, R)$ a W-forest frame, that is, $R$ is such that $(\cleq \circ R) = R = (\cgeq \circ R)$. This means that if $xRy$ and either $z \leq x$ or $z \geq x$, then $zRy$ as well.  This implies that, for any $a \in G(X)$,  if $x\in \beta_R(a)$ (resp. $x \in \delta_R(a)$)  then $y \in \beta_R(a)$ (resp. $y \in \delta_R(a)$) for any $y \in X$ such that $y \leq x$ or $y \geq x$.  In other words, for any $a$, $\beta_R(a)$ and $\delta_R(a)$ are both a downset and an upset. Since any subset of a forest that is both downwards and upwards closed must be a union of a collection of maximal trees of the forest. But maximal trees correspond to joint-irreducible Boolean elements in the algebra 
$\bf G(F)$, and therefore, for any $a \in G(X)$, $\beta_R(a)$ and $\delta_R(a)$ must be Boolean elements of $\bf G(F)$. 

(2) Let $f, g \in F_A$ such that $f R_A g$, that is,  such that $\Box^{-1}(f) \subseteq g$ and $g \subseteq \Diamond^{-1}(f)$. We have to show that if $f' \in F_A$ is such that $f' \subseteq f$ or $f \subseteq f'$, then $f' R_A g$ as well. 

Suppose $f \subseteq f'$. Then clearly, $g \subseteq \Diamond^{-1}(f')$, so let us how that $\Box^{-1}(f') \subseteq g$ as well. By definition, $\Box^{-1}(f') = \{ y \in G(X) \mid \Box(y) \in f'\}$. But by assumption every such $\Box(y)$ is a Boolean element of the prime filter $f'$, and hence  $\Box(y) \in f$ as well. Indeed, by contradiction, suppose $\Box(y)\not \in f$. Then, since $f$ is prime, $\neg \Box(y) \in f$, and thus  $\neg \Box(y) \in f'$ as well, that is a contradiction with the fact that $\Box(y) \in f'$. Therefore, $\Box^{-1}(f') = \Box^{-1}(f) \subseteq G$ and thus  $f' R_A g$. 

The case  $f' \subseteq f'$ can be proved in a similar way. 
\end{proof}

\subsection{A final comparison}

It is now convenient to summarize how the relational frames and their associated classes of G\"odel algebras with operators relate each other.

First of all, notice that the two subvarieties $\mathbb{DGAO}$ and $\mathbb{FSGAO}$ have a non empty intersection as, for instance, the variety $\mathbb{BAO}$ of Boolean algebras with operators is a subvariety of their intersection $\mathbb{FSDGAO} = \mathbb{DGAO}\cap\mathbb{FSGAO}$. Moreover, $\mathbb{DGAO}$ and $\mathbb{FSGAO}$ can be distinguished as we showed in Section \ref{extensions}.

As for the variety $\mathbb{WGAO}$ that we introduced in the above Subsection \ref{MDGAO}, the next result is going to make clear how it relates with the aforementioned varieties as depicted in Figure \ref{figFinal}.

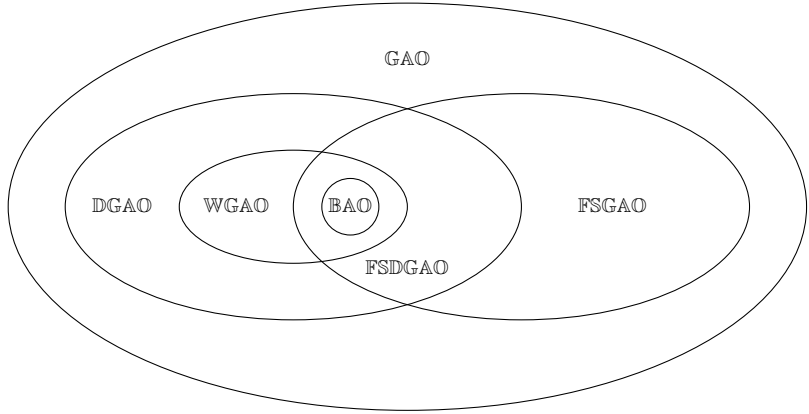
\begin{figure}[h]
\begin{center}
\begin{tikzpicture}[scale=1.5,every node/.style={scale=1.5}]
 \draw (0,0) ellipse (2cm and 1cm);
  \draw (0,0) ellipse (1cm and 0.5cm);
  \draw (2,0) ellipse (2cm and 1cm);
   \draw (0.5,0) ellipse (0.25cm and 0.25cm);
   \draw (1,0) ellipse (3.5cm and 1.8cm);
   \node[] at (2.8,0) {\tiny $\mathbb{FSGAO}$};
   \node[] at (1,-0.55) {\tiny $\mathbb{FSDGAO}$};
    \node[] at (0.5,0) {\tiny $\mathbb{BAO}$};
     \node[] at (-0.5,0) {\tiny $\mathbb{WGAO}$};
      \node[] at (-1.5,0) {\tiny $\mathbb{DGAO}$};
            \node[] at (1,1.3) {\tiny $\mathbb{GAO}$};
\end{tikzpicture}
\end{center}
\caption{A diagram explaining the inclusions among the subvarieties of $\mathbb{GAO}$ that we have studied in the present paper.}\label{figFinal}
\end{figure}

\begin{proposition}\label{propFinal}
The following properties hold:
\begin{itemize}
\item[(i)] $\mathbb{WGAO}\subsetneq \mathbb{DGAO}$,
\item[(ii)] $\mathbb{WGAO}\cap \mathbb{FSDGAO}\neq \emptyset$,
\item[(iii)] $\mathbb{WGAO}\cap \mathbb{FSDGAO}\subsetneq \mathbb{FSDGAO}$,
\item[(iv)] $\mathbb{WGAO}\setminus\mathbb{FSGAO}\neq\emptyset$,
\item[(v)] $\mathbb{BAO}\subsetneq (\mathbb{WGAO}\cap\mathbb{FSGAO})$. 
\end{itemize}
\end{proposition}
\begin{proof}
(i) By Remark \ref{remCJW} (i), every W-forest frame satisfies the conditions that characterize Dunn GAOs. Therefore, $\mathbb{WGAO}\subseteq \mathbb{DGAO}$. So as to prove that the inclusion is proper, consider the tree $\bF$ as in Figure \ref{figTree1} with a relation $R=\{(x,y), (y,y), (z, k), (k, z)\}$. Then $(\bF, R)$ is a CJ-forest frame. Indeed, $(\geq \circ R)\subseteq (R\circ \geq)$ as $(\geq\circ R)=R\cup\{(z, y), (k, y)\}$ and $(R\circ \geq)=R\cup\{(y,x), (z, y), (z, x), (k, y), (k, x), (x,x)\}$. Similarly, $(\leq \circ R)\subseteq (R\circ\leq)$, because $(\leq \circ R)=R\cup\{(y, k), (x, k), (y, z), (x, z)\}$ and $(R\circ \leq)=R\cup\{(x, k), (x, z), (y, k), (y,z)\}$. On the other hand, in the corresponding GAO   $({\bf G}(\bF), \beta, \delta)$, whose G\"odel algebra is that one as in Figure \ref{figFSnotD}, one has that $\beta(\{x, y\})=\{x, y\}$ that is not Boolean and hence $({\bf G}(\bF), \beta, \delta)\not\in \mathbb{WGAO}$. 

(ii) follows because Boolean algebra with operators are, at the same time, a proper subvariety of both $\mathbb{WGAO}$ and $\mathbb{FSDGAO}$. 

As for (iii), consider the finite forest $\bF$ being the tree as in Figure \ref{figTree1} and whose associated G\"odel algebra is as in Figure \ref{figFSnotD}.  Let $R$ be the following relation on $F$: $\{(x, x), (x, y), (y, x), (y,y), (y, z), (k, x), (k, y), (k, z), (z, x), (z, y), (z,z)\}$. Then one can check that $(\leq \circ R)\subseteq (R\circ \leq)$ and $(R\circ \geq)=(\geq \circ R)$. In other words $(\bF, R)$ is a FSD-forest frame and hence its associated GAO $({\bf G}(\bF), \beta, \delta)$ belongs to $\mathbb{FSDGAO}$ (recall Subsection \ref{DGAO}). However, $\beta(\{x, y\})=\{x\}$ in $G(\bF)$ and $\{x\}$ is not Boolean. Therefore, $({\bf G}(\bF), \beta, \delta)\not\in \mathbb{WGAO}$ by definition. 

In order to prove that (iv) holds, consider the algebra  of Figure \ref{figDnotFS} and recall that it does not belong to $\mathbb{FSGAO}$. Furthermore, notice that it belongs to $\mathbb{WGAO}$ as it satisfies (BB) and (DB) of Definition \ref{defWGAO}.

Finally, in order to prove (v), i.e., that $\mathbb{BAO}$ is strictly contained in both $\mathbb{WGAO}$ and $\mathbb{FSGAO}$, let us consider the three element G\"odel chain on domain $A=\{\bot, a, \top\}$ together with the following modal operators: $\Box\top=\Diamond\top=\Diamond a=\top$, $\Box\bot=\Box a=\Diamond \bot=\bot$. Obviously $({\bf A}, \Box, \Diamond)$ is not a BAO and it is a WGAO. Let us hence see that $({\bf A}, \Box, \Diamond)$ satisfies (FS2). For all $x, y\in A$ such that $x\leq y$, $x\to y=\top$ and $\Box\top=\top$. Thus, in these cases (FS2) holds. The remaining three cases are the following: (1) $\Box(\top\to a)=\Box a=\bot$ and,  in this case, $\Diamond \top\to \Box a = \top\to \bot=\bot$, whence (FS2) holds; (2) $\Box(a\to \top)=\bot$; (3) $\Box(\top\to \bot)=\bot$. As for (2), notice that $\Diamond a=\top$ and $\Box\bot=\bot$. Thus $\Diamond a\to \Box \bot=\bot$. In case (3), $\Diamond \top=\top$ and hence $\Diamond \top\to \Box\bot=\bot$. Thus, (FS2) holds in $({\bf A}, \Box, \Diamond)$.
\end{proof}

\section{Conclusions and future work}
In this paper we have been concerned with finite G\"odel modal algebras from several varieties and their corresponding classes of forest frames, which are their dual relational structures.  In particular, we have first considered the more basic class of G\"odel algebras with operators (GAOs) where there is no interaction between the operators and proved a representation theorem in terms of algebras defined on the set of downsets of their prime spectra. The dual relational structures based on forests, called forest frames, are defined by two independent binary relations, one per each modal operator, that satisfy (anti)monotonicity properties on the first argument of them. Then we have considered two subvarieties of G\"odel modal algebras where the operators are not independent any longer, namely those satisfying the so-called Dunn's axioms of positive modal logic and Fischer Servi's axioms for intuitionistic modal logic. For these algebras, the associated dual forest frames are basically specified by a single relation that accounts for the relationship among the operators. In this aspect, we have essentially adapted available results in the general case of Heyting algebras with operators \cite{CelJan,Ale,Orlo} to our case and have further investigated on the relations between the additional axioms and properties of the relations in the forest frames. Finally we have considered two further subvarieties of G\"odel modal algebras, the one whose algebras satisfy both Dunn and Fischer Servi axioms and the one  whose modal operators always yield Boolean elements, and their corresponding  forest frames have a single binary relation satisfying both monotonicity and antimonotonicity properties in the first argument.

As for future work, there are at least a couple of interesting issues that  deserve further research. 
A first issue is the relationship between the forest-based relational semantics for G\"odel modal logics, considered in this paper along the line of intuitionistic modal logics, and the $[0, 1]$-valued Kripke models that have been used in other venues like e.g. in \cite{XARO,BEGR, RodVid} following the strand of fuzzy modal logics. It is not clear whether there exists a direct relationship between them. On the one hand, the minimal modal logic complete with respect to $[0, 1]$-valued Kripke models, i.e. the bimodal  G\"odel logic studied in \cite{XARO}, already satisfies the Fischer Servi axioms. So it seems that $[0, 1]$-valued Kripke models only account for G\"odel modal logics  satisfying both Fischer Servi axioms. Thus, for instance $[0, 1]$-valued Kripke models seems not to be a semantics for the logic of our GAOs algebras nor  those extended with axioms (D1), (BB) and (DB).

A second issue is to extend our research on G\"odel modal algebras over related  algebraic structures with nice duality theories. A clear candidate is the variety of Nilpotent Minimum algebras (NM-algebras for short) that, similarly to G\"odel algebras, is dual to a forest-based category (see \cite{BuCi}). Another class of algebras that is dual to forests is that of IUML-algebras, studied in \cite{ABCG}. The latter are algebras based on uninorms rather than  t-norms, but similar methods to those developed in the present chapter might be used to define and study modal operators on them.

%
%



\begin{thebibliography}{99}


\bibitem{ABCG}
S. Aguzzoli, S. Boffa, D. Ciucci, B. Gerla. Finite IUML-algebras, Finite Forests and Orthopairs. 
{\em Fundamenta Informaticae} 163(2):139--163, 2018.

\bibitem{ABG}
S. Aguzzoli, S. Bova, B. Gerla. Free Algebras and Functional Representation for Fuzzy Logics. Chapter IX of the {\em Handbook of Mathematical Fuzzy Logic} - Volume 2. P. Cintula, P. H\'ajek, C. Noguera Eds., Studies in Logic, vol. 38, College Publications, London, pp. 713--791, 2011.

\bibitem{AFU}
S. Aguzzoli, T. Flaminio, S. Ugolini. Equivalences between subcategories of MTL-algebras via Boolean algebras and prelinear semihoops. {\em Journal of Logic and Computation} 27(8): 2525--2549, 2017. 

\bibitem{BaazPre}
M. Baaz, N. Preining. G\"odel-Dummett Logics.
Chapter VII of the {\em Handbook of Mathematical Fuzzy Logic} - Volume 2. P. Cintula, P. H\'ajek, C. Noguera Eds., Studies in Logic, vol. 38, College Publications, London, pp. 713--791, 2011.


\bibitem{BRV}
 P. Blackburn, M. de Rijke, Y. Venema. {\em Modal Logic}.  Cambridge University Press, 2001.

\bibitem{BEGR}
F. Bou, F. Esteva, L. Godo, R. O. Rodriguez. On the Minimum Many-Valued Modal Logic over a Finite Residuated Lattice. {\em Journal of Logic and Computation} 21(5): 739--790, 2011.

\bibitem{BoDo} M. Bo\v{z}i\'c, K. Do\v{s}en. Models for normal intuitionistic modal logics. {\em Studia Logica} 43: 217--245,1984.

\bibitem{BuCi}
M. Busaniche R. Cignoli. Constructive Logic with Strong Negation as a Substructural Logic. {\em Journal of Logic and Computation} 20(4): 761--793, 2010.

\bibitem{CMRR}
X. Caicedo, G. Metcalfe, R.O. Rodriguez, J. Rogger. A Finite Model Property for G\"odel Modal Logics. In: Libkin L., Kohlenbach U., de Queiroz R. (eds) Logic, Language, Information, and Computation. WoLLIC 2013. Lecture Notes in Computer Science, 8071, 2013. 

 \bibitem{XARO}
 X. Caicedo, R. O. Rodriguez. Bi-modal G\"odel logic over $[0, 1]$-valued Kripke frames. {\em Journal of Logic and Computation} 25(1): 37--55, 2015.



\bibitem{Cel01} S. Celani.  Remarks on Intuitionistic Modal Logics. {\em Divulgaciones Matem\'aticas} 9 (2): 137--147, 2001. 

\bibitem{Cel06} S. Celani.  Simple and subdirectly irreducibles bounded distributive lattices with unary operators. {\em International Journal of Mathematics and Mathematical Sciences}. Article ID 21835, 20 pages, 2006.

\bibitem{Cel08} S. Celani. Notes on the representation of Distributive Modal algebras. {\em Miskolc Mathematical Notes} 9 (2): 81--89, 2008. 


\bibitem{CelJan}
S. Celani, R. Jansana. Priestley Duality, a Sahlqvist Theorem and a Goldblatt-Thomason Theorem for Positive Modal Logic. {\em Logic Journal of the IGPL} 7(6): 683--715, 1999.

%
%

\bibitem{DMS}
D. Diaconescu, G. Metcalfe, L. Schn\"uriger. A Real-Valued Modal Logic. {\em Logical Methods in Computer Science} 14(1): 1--27, 2018.

\bibitem{Dosen} K. Do\v{s}en. Models for Stronger Normal Intuitionistic Modal Logics. {\em Studia Logica} 44: 39--70, 1985. 

\bibitem{Dunn}
M. Dunn, Positive Modal Logics. {\em Studia Logica} 55: 301--317, 1995.

\bibitem{FS}
G. Fischer Servi. Axiomatizations for some intuitionistic modal logics. {\em Rend. Sem. Mat. Polit de Torino} 42, 179--194, 1984.

\bibitem{FGR}
T. Flaminio, L. Godo, R. O. Rodriguez. A representation theorem for finite G\"odel algebras with operators. In: Iemhoff R., Moortgat M., de Queiroz R. (eds). Logic, Language, Information, and Computation, WoLLIC 2019. LNCS 11541: 223--235, Springer, 2019. 

\bibitem{Goldblatt} R. Goldblatt. Varieties of complex algebras. {\em Annals of Pure and Applied Logic} 44 (3): 173--242, 1989. 

 \bibitem{H98}
 P. H\'ajek. {\em Metamathematics of Fuzzy Logic}. Kluwer Academic Publishers, 1998.

 \bibitem{Hasimoto2001}
 Y. Hasimoto. Heyting algebras with operators. {\em Mathematical Logical Quarterly} 47(2): 187--196, 2001.

\bibitem{Horn}
A. Horn. Logic with truth values in a linearly ordered Heyting algebra, {\em The Journal of Symbolic Logic} 34: 395--405, 1969.

\bibitem{JT61}
B. J\'onsson and A. Tarski. Boolean algebras with operators. Part I. {\em American Journal of Mathematics} 73(4): 891--939, 1951. 


\bibitem{Lemmon}
E. J. Lemmon. Algebraic semantics for modal logics I. {\em Journal of Symbolic Logic} 31: 45--65, 1966.

\bibitem{Palmigiano2}
M. Ma, A. Palmigiano, M. Sadrzadeh. Algebraic semantics and model completeness for Intuitionistic Public Announcement Logic. {\em Annals of Pure and Applied Logic} 165: 963--995, 2014.
 
\bibitem{Ono} H. Ono. On Some Intuitionistic Modal Logics. {\em Publication of the Research Institute for Math. Sc.} 13: 687--722, 1977.

 
\bibitem{Orlo} E. Or{\l}owska, I. Rewitzky. Discrete Dualities for Heyting Algebras with Operators. {\em Fundamenta Informaticae} 81: 275--295, 2007. 

\bibitem{Ale} A. Palmigiano. Dualities for Intuitionistic Modal Logics. In {\em Liber Amicorum for Dick de Jongh},  Institute for Logic, Language and Computation, University of Amsterdam,  pp.  151-167, 2004. http://festschriften.illc.uva.nl/D65/palmigiano.pdf. 


\bibitem{Petrovich} A. Petrovich. Distributive lattices with an operator. {\em Studia Logica} 56 (1-2): 205--224, 1986.

\bibitem{PlotStir} G. Plotkin, C. Stirling. A Framework for Intuitionistic Modal Logic. In  J. Y. Halpern (ed.), Proceedings of the 1st Conference on Theorical Aspects of Reasoning and Knowledge, 399--406, Morgan-Kaufmann, 1986.

\bibitem{RodVid}
R. O. Rodriguez, A. Vidal. Axiomatization of Crisp G\"odel Modal Logic. {\em Studia Logica} 109(2): 367--395, 2021. 

\bibitem{Viorica00} V. Sofronie-Stokkermans. Duality and canonical extensions of bounded distributive lattices with operators, and applications to the semantics of non-classical logics I. {\em Studia Logica} 64 (1): 93--132, 2000. 

\bibitem{Viorica00-2} V. Sofronie-Stokkermans. Duality and canonical extensions of bounded distributive lattices with operators, and applications to the semantics of non-classical logics II. {\em Studia Logica} 64 (1): 151--172, 2000. 

\bibitem{Viorica03} V. Sofronie-Stokkermans. Representation theorems and the semantics of non-classical logics, and applications to automated theorem proving. In Beyond Two: Theory and Applications of Multiple-Valued Logic, (M. Fitting and E. Or{\l}owska, eds.), Studies in Fuzziness and Soft Computing, vol. 114, Physica, Heidelberg, 59--100, 2003. 


\bibitem{Sotirov} V.H. Sotirov. Modal Theories with Intuitionistic Logic. Proceedings of the Conference on Mathematical Logic, Sofia (Bulgaria), Bulgarian Academy of Sciences, 139-171, 1984. 

\bibitem{Take}
G. Takeuti, S. Titani. Intuitionistic fuzzy logic and intuitionistic fuzzy set theory. {\em Journal of Symbolic Logic}, 49(3):851--866, 1984.

\bibitem{Wolter0}
F. Wolter, M. Zakharyaschev. The relation between intuitionistic and classical modal logics. {\em Algebra and Logic} 36: 73--92, 1997.

\bibitem{Wolter}
F. Wolter, M. Zakharyaschev. Intuitionistic Modal Logic. In: A. Cantini, E. Casari, P. Minari (eds) {\em Logic and Foundations of Mathematics}. Synthese Library (Studies in Epistemology, Logic, Methodology, and Philosophy of Science), vol 280. Springer, Dordrecht, 1999. 
\end{thebibliography}
\end{document}